\PassOptionsToPackage{reqno}{amsmath}
\documentclass[pdftex]{amsart}

\title{Semitoric systems with $2n$ Focus-Focus singularities}
\author{Haruhisa Mura}
\date{}
\usepackage{amssymb}
\usepackage{amsthm}
\usepackage{hyperref
}
\usepackage{stmaryrd}
\usepackage{physics}
\usepackage{tikz}
\usepackage{tikz-cd}
\usepackage{graphicx}

\makeatother
\usepackage{nicematrix}
\usepackage{enumitem}
\usepackage{etoolbox}
\usepackage[normalem]{ulem}
\usepackage{bm}
\usepackage{arydshln}
\numberwithin{equation}{section}
\newcommand{\itemeq}{
\refstepcounter{equation}
\hfill (\theequation)
}

\newtheorem{theorem}{Theorem}[section]
\newtheorem{lemma}[theorem]{Lemma}
\theoremstyle{definition}
\newtheorem{definition}[theorem]{Definition}

\newtheorem{proposition}[theorem]{Proposition}
\theoremstyle{remark}

\NiceMatrixOptions{xdots/line-style=dotted}
\NiceMatrixOptions{
  custom-line = {
    command = myhdottedline, 
    letter = :,              
    tikz = {                 
      line cap=round, 
      dash pattern=on 0pt off 2pt, 
      line width=0.6pt
    }
  }
}

\allowdisplaybreaks
\begin{document}
\begin{abstract}
In this paper, we provide new explicit examples of semitoric systems. More precisely, for any natural number $n$, we give a semtioric moment map $F_n$ with $2n$ focus-focus singularities defined on a symplectic manifold $(M_n,\omega_n)$ of diffeomorphism type $\mathbb{CP}^2\sharp(2n+1)\overline{\mathbb{CP}}{}^2$.
\end{abstract}
\maketitle

\section{Introduction}
An \textit{integrable Hamiltonian system} of $m$-degrees of freedom (hereafter referred to as an integrable system) is a pair consisting of a $2m$-dimensional symplectic manifold $(M,\omega)$ and a map $F:M\longrightarrow\mathbb{R}^m$ called a moment map, where the moment map $F$ is regular almost everywhere, and all of its component functions are pairwise Poisson commutative. Classically known examples of integrable systems include the Kepler problem, the two-body problem in celestial mechanics. In quantum mechanics, examples include the coupled angular momenta and the Jaynes-Cummings model. (For details, see, for example, \cite{babelon2009semi}, \cite{perelomov1990integrable}, and \cite{sadovskii1999monodromy}.)

When the moment map is proper, the set of flows generated by the Hamiltonian vector fields of each component function defines an $\mathbb{R}^m$-action on the symplectic manifold. An integrable system on a compact symplectic manifold is called a \textit{toric system} if the $\mathbb{R}^m$-action induces an effective $\mathbb{R}^m/\mathbb{Z}^m$-action. A symplectic manifold equipped with the structure of a toric system is called a \textit{symplectic toric manifold} (hereafter referred to as a \textit{toric manifold}). 
In the 1980s, building on the convexity theorem by Atiyah \cite{Atiyah} and Guillemin-Sternberg \cite{GScovexity}, Delzant \cite{Delzant} demonstrated a one-to-one correspondence between the isomorphism classes of toric systems and the isomorphism classes of Delzant polytopes. Now, this theorem is known as Delzant theorem, and it connects K\"{a}hler geometry and symplectic geometry with combinatorics. (For instance, Hamilton \cite{geomquantiontoric} discussed the relationship between the geometric quantization of toric manifolds and the number of lattice points in the corresponding Delzant polytopes.)

In the late 2000s, Pelayo and V{\~u} Ng{\d{o}}c \cite{pelayo2009semitoric} introduced \textit{semitoric systems} as a generalization of toric systems. Roughly speaking, 
a semitoric system is an integrable system of $2$-degrees of freedom $(M,\omega,F=(J,H):M\longrightarrow\mathbb{R}^2)$ that satisfies the following conditions:
\begin{itemize}
\item The function $J$ is proper, and the flow generated by the Hamiltonian vector field of $J$ is periodic;
\item
The moment map $F$ admits the singularities of toric moment maps and \textit{focus-focus} singularities.
\end{itemize}
For the precise definition, see Definition \ref{dfn:semitoric}.
Pelayo and V{\~u} Ng{\d{o}}c \cite{pelayo2009semitoric} revealed that semitoric systems of a special type, called \textit{simple} semitoric systems, are classified by five invariants. Subsequently, they constructed all simple semitoric systems (up to isomorphism) from these invariants in \cite{pelayo2011constructing}. Furthermore, in \cite{palmer2024semitoric}, they extended these results to the general case, which is not restricted to the simple type. 

One of the crucial points of Delzant theorem is that the method for reconstructing a toric system from a Delzant polytope is extremely simple and explicit. This method is now called the  Delzant construction, and it can explicitly describe all toric systems (up to isomorphism). On the other hand, the construction methods for semitoric systems in \cite{palmer2024semitoric} and \cite{pelayo2011constructing} rely on the gluing of local models; therefore, in contrast to the Delzant construction, the resulting semitoric systems do not have explicit representations. Moreover, explicit examples of semitoric systems with focus-focus singularities are relatively scarce (specific examples with one focus-focus singularity can be found in \cite{chiscop2019lagrangian}; with two in \cite{semitoricex2},\,\cite{semitoricex1},\,\cite{le2019symplectic},\,\cite{pentagonspaces},\,and \cite{sadovskii1999monodromy}; and with four in \cite{Octagon}).

The purpose of this paper is to provide new explicit examples of semitoric systems with focus-focus singularities. Here, we outline the construction of concrete examples. First, for any natural number $n$, we construct a Delzant polytope $\Delta_n$ with $2n+4$ vertices. For the precise definition of $\Delta_n$, see (\ref{Delta_n}). 
Then, we apply the Delzant construction to $\Delta_n$ to obtain a toric system $(M_n, \omega_n, \mu_n)$.  
Here, the toric manifold $(M_n, \omega_n)$ is constructed as a quotient of a certain subset $v^{-1}(0)$ of $\mathbb{C}^{2n+4}$, and its diffeomorphism type is the $(2n+1)$-fold blow-up of $\mathbb{CP}^2$. Next, we briefly outline the idea for constructing the desired moment map $F_n:M_n\longrightarrow\mathbb{R}^2$. We remark that this idea is a generalization of the one found in the literature \cite[Chapter~6]{semitoricfamilies}. 
We define a complex-valued function $f_n$ on $\mathbb{C}^{2n+4}$ by $f_n=z_1\cdots z_{2n}\bar{z}_{2n+3}\bar{z}_{2n+4}$.
Then, $f_n$ induces a function $\widehat{f}_n$ on $M_n$. The function $f_n$ is ``degenerate'' at points in $\mathbb{C}^{2n+4}$ with three or more zero coordinates, but the set $v^{-1}(0)$ does not contain such points. Therefore, $f_n$ is ``non-degenerate'' on the set $v^{-1}(0)$, and thus $\widehat{f}_n$ is also ``non-degenerate''. Now, we choose an appropriate Hamiltonian $S^1$-function $J_n$ using the toric moment map $\mu _n$. such that $f_n$ is invariant under the Hamiltonian $S^1$-action generated by $J_n$. Finally, we define a map $F_n$ by $F_n=(J_n, \text{Re}\,\widehat{f}_n)$. By examining the types of singularities using the method of Lagrange multipliers, we obtain the following result.

\begin{theorem}(Theorem~\ref{mainthm})
  For any natural number $n$, the tuple $(M_n, \omega_n, F_n)$ is a semitoric system with $2n$ focus-focus singularities.
\end{theorem}

\subsection*{Structure of the paper}
\begin{description}
\item[Section~\ref{chapter2}] First, we recall the definitions of basic concepts and  important theorems. At the end of this section, we briefly review the Delzant construction.

\item[Section~\ref{chapter3}] We explicitly construct the Delzant polytope $\Delta_n$. Then, we apply the  Delzant construction to $\Delta_n$ to obtain the toric manifold $(M_n, \omega_n)$.

\item[Section~\ref{chapter4}] 
In this section, we define the moment map $F_n=(J_n,H_n):M_n\longrightarrow\mathbb{R}^2$ on the toric manifold $(M_n,\omega_n)$, and prove the main result (Theorem~\ref{mainthm}) assuming Proposition~\ref{prop:singtype}.
  \item[Section~\ref{chapter5}] In this final section, we prove Proposition \ref{prop:singtype}. Namely, we examine the singularities of the map $F$ and their types.
\end{description}

\subsection*{Acknowledgements} The author would like to express sincere gratitude to Professor Takahiro Oba for  providing numerous helpful suggestions regarding the writing of this paper. This work was supported by JST SPRING, Grant Number JPMJSP2138.

\section{Fundamental Definitions and Important Results}\label{chapter2}
In this section, we recall some of the fundamental definitions and results related to integrable systems. 
\begin{definition}
Let $(M,\omega)$ be a $2m$-dimensional symplectic manifold and $F=(f_1,\ldots,\allowbreak f_m):M\longrightarrow\mathbb{R}^{m}$ be a smooth map. The triple $(M, \omega, F)$ is called an \textbf{integrable system} with $m$-degrees of freedom if the map $F$ is regular on a dense subset of $M$, and the Poisson bracket of $f_i$ and $f_j$ vanishes for all $i,j=1,\ldots,m$. The map $F$ is called the \textbf{moment map}.
\end{definition}
In what follows, we assume that all moment maps are proper.

A representative class of integrable systems is the toric system. Let us recall its definition. In the following, let $\mathfrak{t}$ be the Lie algebra of a torus $T$.
\begin{definition}
Let $(M,\omega,T,\mu)$ be a \textbf{Hamiltonian $T$-space}. That is, a torus $T$ acts symplectically on a symplectic manifold $(M,\omega)$, and the map $\mu:M\longrightarrow \mathfrak{t}^*$, called the \textbf{moment map}, satisfies the following conditions:
  \begin{itemize}
    \item For any $X\in \mathfrak{t}$, it holds that $d\langle \mu, X\rangle =\iota_{X^\#}\omega$, where $X^\#$ is the fundamental vector field on $M$ generated by $X$;
    \item The map $\mu$ is $T$-invariant.
  \end{itemize}
In this case, the torus action is called a \textbf{Hamiltonian $T$-action}.
A Hamiltonian $T$-space $(M,\omega,T, \mu)$ is called a \textbf{toric system} if the torus action is effective and the dimension of the acting torus is half that of $M$. We often omit the torus $T$ from the notation for a toric system as it is uniquely determined. A symplectic manifold that admits the structure of a toric system is called a \textbf
{toric manifold}.
\end{definition}
Let $(M,\omega, F)$ be an integrable system. A point $p\in M$ is a \textbf{singularity} if the differential of $F$ at $p$, denote by $dF_p$, does not have maximal rank.
Next, we recall the nondegeneracy and the types of singularities in integrable systems. 
Let $p\in M$ be a singularity of rank $i$. Then, there exists $Q\in \text{GL}(m, \mathbb{R})$ such that by setting $(g_1, g_2, \ldots, g_m) = Q\circ (f_1, f_2, \ldots, f_m)$, it holds that $dg_1(p) = dg_2(p) = \cdots = dg_{m-i}(p) = 0$. In this case, the functions $g_1,\ldots,g_{m-i}$ define the following isotropy representation:
\[
\begin{tikzcd}[nodes={anchor=base}]
   \Psi:&[-30pt]\mathbb{R}^{m-i} \arrow[r] & \text{Sp}(T_pM, \omega_p) \\[-2ex]
   & (t_1, t_2, \ldots, t_{m-i}) \arrow[u, "\in" {rotate=90, anchor=center}, phantom] \arrow[r, "\longmapsto", phantom] & \left(d\psi^{g_1}_{t_1}\right)_p \circ \cdots \circ \bigl(d\psi^{g_{m-i}}_{t_{m-i}}\bigr)_p\text{\raisebox{-5pt}{,}} \arrow[u, "\in" {rotate=90, anchor=center}, phantom]
\end{tikzcd}
\]
where $\psi^{g_j}_{t_j}$ is the symplectomorphism on $M$ obtained as the time-$t_j$ map of the flow generated by the Hamiltonian vector field of $g_j$.
Now, for $j=1,\ldots,m-i$, let $A_j \in \text{End}(T_pM)$ be defined as $A_j = d\Psi_0\left(\left. \pdv{}{t_j}\right|_0\right)$. Then, on a chart $(U, \phi)$ containing the point $p$, it holds that $A_j = d_p^2G_j\cdot\Omega^{-1}_p$. Here, $d_p^2G_j$ and $\Omega_p$ are the matrix representation of the Hessian $d_p^2g_j$ and the symplectic form $\omega_p$ on the chart $(U,\phi)$, respectively. Let $L_p\subset T_pM$ be the subspace of $T_pM$ defined as $L_p = \langle X_{g_{m-i+1}}(p), X_{g_{m-i+2}}(p), \ldots, X_{g_m}(p)\rangle$. Then, since the Hessians $d^2_pg_j$ vanishes on $L_p$, it induces a bilinear form $\widehat{d^2_pg_j}$ on $L_p^\omega/L_p$. On the other hand, since $L_p$ is an isotropic subspace of $(T_pM,\omega_p)$, it naturally induces a linear symplectic form $\widehat{\omega}_p$ on $L_p^\omega/L_p$.
\begin{definition}\label{def:nondegeneracy}
A singularity $p$ of rank $i$ of an integrable system $(M,\omega,F)$ with $m$-degrees of freedom is said to be \textbf{non-degenerate} if it satisfies the following two conditions:
  \begin{itemize}
    \item The bilinear forms $\widehat{d_p^2g_1}, \ldots, \widehat{d_p^2g_{m-i}}$ are linearly independent;
    \item There exist real numbers $a_1,\ldots,a_{m-i}$ such that \textbf{the characteristic equation}\vspace{3mm}
  \begin{equation}\label{chreq}
    P(t) = \det(\sum_{j=1}^{m-i}a_j\widehat{d_p^2G_j}\widehat{\Omega}^{-1}_p-tI) = 0\\[2ex]
  \end{equation}
  has $2(m-i)$ non-zero roots. Here, $\widehat{d_p^2G_j}$ and $\widehat{\Omega}_p$ are representation matrices of $\widehat{d_p^2g_j}$ and $\widehat{\omega}_p$, respectively and $I$ is the identity matrix. 
  \end{itemize}
\end{definition}
If a singularity $p$ is non-degenerate, the solution set of the characteristic equation (\ref{chreq}) is divided into the following three types of blocks:
\begin{itemize}
\item \textbf{(Elliptic block)} a pair of purely imaginary roots $(i\alpha, -i\alpha)$; \notag
\item \textbf{(Hyperbolic block)} a pair of real roots $(\beta, -\beta)$; \itemeq \label{chtype1}
\item \textbf{(Focus-Focus block)} a quadruple $(\alpha+i\beta, \alpha-i\beta, -\alpha+i\beta, -\alpha-i\beta)$.\notag 
\end{itemize}
\begin{definition}\label[(Pelayo, V\~{u} Ng\d{o}c \cite{pelayo2009semitoric})\label{dfn:semitoric}
  A \textbf{semitoric system} is an integrable system with $2$-degrees of freedom $(M,\omega, F=(J,H))$ which satisfies the following conditions:
  \begin{itemize}
    \item The maps $F$ and $J$ are proper;
    \item The function $J$ generates an effective Hamiltonian $S^1$-action;
    \item All singularities are non-degenerate, and there are no hyperbolic blocks among the solutions to their characteristic equations.
  \end{itemize}
\end{definition}
In the $2$-degrees of freedom integrable systems, non-degenerate singularities with no hyperbolic blocks
are classified into the following three types:
\begin{itemize}
  \item \textbf{ER(Elliptic-Regular)} rank\,$1$ with an elliptic block;
  \item \textbf{EE(Elliptic-Elliptic)} rank\,$0$ with two elliptic blocks; \itemeq \label{chtype2}
  \item \textbf{FF(Focus-Focus)} rank\,$0$ with a focus-focus block.
\end{itemize}
As shown in the following theorem, if a non-degenerate singularity has no hyperbolic blocks,
 a neighborhood of the singularity can be identified with a standard model determined by the types of blocks in the solution set of the characteristic equation. Since this paper focuses on integrable systems with $2$-degrees of freedom, we restrict the statement of the orem to that case. The statement for general degrees of freedom is analogous; see for example, \cite{eliasson1990normal}.
\begin{theorem}\label[(\cite{eliasson1990normal})
Let $(M,\omega,F)$ be an integrable system with $2$-degrees of freedom, and let $p\in M$ be a non-degenerate singularity with no hyperbolic blocks. Then, there exist a symplectic coordinate neighborhood $\bigl(U,\omega = dp_1\wedge dq_1+dp_2\wedge dq_2;$\linebreak$\Phi=(p_1,q_1,p_2,q_2)\bigr)$ centered at $p$ and a diffeomorphism $\phi:\mathbb{R}^2(F(p))\longrightarrow \mathbb{R}^2(0)$ such that $\phi \circ F = (g_1, g_2)\circ \Phi$. Here, $\mathbb{R}^2(F(p))$ is an open neighborhood of $F(p)$, $\mathbb{R}^2(0)$ is an open neighborhood of $0\in \mathbb{R}^2$, and $(g_1,g_2)$ is given as follows:  
\begin{itemize}
\item If $p$ is of ER type, then $(g_1, g_2) = (q_1,\,p_2^2+q_2^2)$;
\item If $p$ is of EE type, then $(g_1, g_2) = (p_1^2+q_1^2,\,p_2^2+q_2^2)$;
\item If $p$ is of FF type, then $(g_1, g_2) = (p_1q_1+p_2q_2,\,p_1q_2-p_2q_1)$.
\end{itemize}
\end{theorem}

Before concluding this section, we briefly recall Delzant's construction, a method for constructing a toric manifold from a Delzant polytope. For details, see for example, \cite{Delzant} and \cite{Cannas}.

Let $\Delta \subset \mathbb{R}^m$ be a Delzant polytope, i.e. $\Delta$ is a polytope in $\mathbb{R}^m$ which satisfies the following conditions:
\begin{itemize}
\item(Simplicity) Each vertex of $\Delta$ is the intersection of exactly $m$ edges;
  \item(Rationality) The edges meeting at a vertex $q$ are represented by the form $q+tu_i$ $(t\geq 0, u_i\in \mathbb{Z}^m)$;
  \item(Smoothness) For each vertex $q$, the above integer vectors $u_1,\ldots u_m$ form a basis for $\mathbb{Z}^m$.
\end{itemize}
By identifying $\mathbb{R}^m$ with its dual $(\mathbb{R}^m)^*$, we regard $\Delta$ as a subset of $(\mathbb{R}^m)^*$. Suppose that $\Delta$ has $k$ edges. Then, there exists a unique set of pairs $\{(N_j, c_j)\}_{j=1,\ldots, k}$ of an integer vector and a real number, up to ordering, satisfying the following two conditions:
\begin{itemize}
  \item $\Delta = \{\eta \in (\mathbb{R}^m)^*\mid \langle\eta, N_j \rangle \leq c_j, \text{ for all } j=1,\ldots, k\}$; 
  \item Each $j=1,\ldots, k$ and $m\in \mathbb{Z}$, $N_j/m\in \mathbb{Z} \iff m = \pm 1$.
\end{itemize}
Let $\Pi:\mathbb{R}^k\longrightarrow \mathbb{R}^m$ be the map defined by $\Pi(e_j) = N_j$ $(j=1,\ldots, k)$ where $e_1, \ldots, e_k$ are the standard basis for $\mathbb{R}^k$. Note that $\Pi$ is surjective because Delzant polytopes satisfy the smoothness condition. Moreover, since $\Pi(\mathbb{Z}^k)=\mathbb{Z}^m$, $\Pi$ descends to the map $\widehat{\Pi}:\mathbb{R}^k/\mathbb{Z}^k \longrightarrow \mathbb{R}^m/\mathbb{Z}^m$ so that $\exp \circ\, \Pi=\widehat{\Pi} \circ \exp$, where $\exp$ is the exponential map of tori. Let $\mathfrak{k}$ and $K$ be the kernel of $\Pi$ and $\widehat{\Pi}$, respectively. Note that $K$ is a Lie subgroup of $\mathbb{R}^k/\mathbb{Z}^k$ with Lie algebra $\mathfrak{k}$, and is isomorphic to a $(k-m)$-dimensional torus. Therefore, we obtain the following commutative diagram of two short exact sequences:
\begin{equation*}
\begin{tikzcd}
0 \arrow[r] & \mathfrak{k}\arrow[r, "\iota"] \arrow[d, "\exp"] & \mathbb{R}^k \arrow[r, "\Pi"] \arrow[d, "\exp"] & \mathbb{R}^m \arrow[r] \arrow[d, "\exp"] & 0 \\
1 \arrow[r] & K \arrow[r] & \mathbb{R}^k/\mathbb{Z}^k \arrow[r, "\widehat{\Pi}"] & \mathbb{R}^m/\mathbb{Z}^m \arrow[r] & 1
\end{tikzcd}
\end{equation*}
Here, $\iota$ is the inclusion. In addition, 
the dual of the first row in the diagram above yields the following short exact sequence:
\begin{equation}\label{eq:exactseq}
\begin{tikzcd}
0 \arrow[r] & (\mathbb{R}^m)^*\arrow[r, "\Pi^*"] & (\mathbb{R}^k)^* \arrow[r, "\iota^*"] & \mathfrak{k}^* \arrow[r] & 0
\end{tikzcd}
\end{equation}
Next, let $v_0:\mathbb{C}^k\longrightarrow \mathbb{R}^k$ be the map defined by 
\begin{gather*}
  v_0(z_1, \ldots, z_k) =\left(-\dfrac{1}{2}|z_1|^2+c_1, \ldots, -\dfrac{1}{2}|z_k|^2+c_k\right)
  \text{\raisebox{-8pt}{,}}
\end{gather*}
and let $v =\iota^*\circ v_0:\mathbb{C}^k\longrightarrow\mathfrak{k}^*$. Note that $v_0$ is the moment map for the Hamiltonian $\mathbb{R}^k/\mathbb{Z}^k$-action 
\begin{equation}\label{eq:std-action}
   \begin{tikzcd}
    T^k \times \mathbb{C}^k \arrow[r] & \mathbb{C}^k \\[-1ex]
    (t_1, \ldots, t_k, z_1, \ldots, z_k) \arrow[u, "\in"{rotate=90, anchor=center}, phantom] \arrow[r, "\longmapsto", phantom]& (e^{2\pi it_1}z_1, \ldots, e^{2\pi it_k}z_k) \arrow[u, "\in"{rotate=90, anchor=center}, phantom] 
   \end{tikzcd}
\end{equation} 
and $v$ is for the above Hamiltonian action restricted to $K$ and $v^{-1}(0)$.
   Since $\Pi^*$ is injective, a left inverse $(\Pi^*)^{-1}$ of $\Pi^*$ exists. Moreover, since $(\ref{eq:exactseq})$ is exact, we get $v_0(v^{-1}(0)) \subset \Pi^*((\mathbb{R}^m)^*)$, so the map $\mu=(\Pi^*)^{-1} \circ v_0\circ i:v^{-1}(0)\longrightarrow (\mathbb{R}^m)^*$ does not depend on the choice of a left inverse of $\Pi^*$, where $i:v^{-1}(0)\longrightarrow \mathbb{C}^k$ is the inclusion. 
Because the restriction of the action (\ref{eq:std-action}) to $K$ and $v^{-1}(0)$ is free and symplectic, the quotient space $v^{-1}(0)/K$ is a symplectic manifold endowed with a symplectic form $\omega_\Delta$ naturally induced from $\mathbb{C}^k$. Moreover, since the map $\mu$ is $K$-invariant, it induces the map $\widehat{\mu}: v^{-1}(0)/K\longrightarrow (\mathbb{R}^m)^*$ which satisfies $\widehat{\mu} \circ q = \mu$, where $q:v^{-1}(0)\longrightarrow v^{-1}(0)/K$ is the quotient map. In conclusion, the desired toric system is given by $(v^{-1}(0), \omega_\Delta, \widehat{\mu})$. The following diagram summarizes the relationships among the maps:
\begin{equation}\label{diagram:Delzant}
 \begin{tikzcd}
  v^{-1}(0)/K \arrow[rd, "\widehat{\mu}"']& v^{-1}(0) \arrow[l, "q"] \arrow[r, hook, "i"] \arrow[d, "\mu"] & \mathbb{C}^k \arrow[d, "v_0"] \arrow[rd, "v"] \\
  0 \arrow[r] & (\mathbb{R}^m)^*\arrow[r, "\Pi^*"] & (\mathbb{R}^k)^* \arrow[r, "\iota^*"] & \mathfrak{k}^* \arrow[r] & 0
 \end{tikzcd}
\end{equation}


\section{Construction of The Toric Systems \texorpdfstring{$(M_n,\omega_n, \mu_n)$}{}}\label{chapter3}
In this section, for any natural number $n$, we construct a Delzant polytope $\Delta_n$ and apply the Delzant construction to it to obtain a toric system $(M_n,\omega_n,\mu_n)$.
For the Delzant construction and some notation, see the end of the previous section.
  
  Let $\Delta_n$ be the Delzant polytope characterized by the following set of pairs of the integer vector and the real number:
  \begin{equation}\label{Delta_n}
    \Delta_n = \{(N_1^\pm, c_1^\pm), (N_2^\pm, c_2^\pm), \ldots, (N_n^\pm, c_n^\pm), (N_s^\pm, c_s^\pm), (N_x, c_x), (N_y, c_y)\}\subset \mathbb{Z}^2\times\mathbb{R},
  \end{equation}
where\vspace{-2mm}
\begin{align*}
  (N_j^+, c_j^+)&=\Biggl(              
  \begin{pmatrix}
    j-1\\j
  \end{pmatrix}
  \text{\raisebox{-7pt}{,}}
  \dfrac{j(j-1)}{2}\Biggr)\hspace{3mm}(1\leq j \leq n),\\
   (N_j^-, c_j^-)&=\Biggl(             
  \begin{pmatrix}
    -j\\1-j
  \end{pmatrix}
  \text{\raisebox{-7pt}{,}}
  \dfrac{j(j-1)}{2}\Biggr)\hspace{3mm}(1\leq j \leq n),\\
  (N_s^+, c_s^+) &=\Biggl(               
  \begin{pmatrix}
    1\\1
  \end{pmatrix}
  \text{\raisebox{-7pt}{,}}
  n\Biggr)
  \text{\raisebox{-7pt}{,}}
  \\
  (N_s^-, c_s^-) &=\Biggl(               
  \begin{pmatrix}
    -1\\-1
  \end{pmatrix}
  \text{\raisebox{-7pt}{,}}
  n\Biggr) \text{\raisebox{-7pt}{,}}
  \\
  (N_x, c_x) &=\Biggl(                  
  \begin{pmatrix}
    1\\0
  \end{pmatrix}
  \text{\raisebox{-7pt}{, }}
  1+\dfrac{n(n+1)}{2}\Biggr)
  \text{\raisebox{-7pt}{, }}
  \\
  (N_y, c_y) &=\Biggl(                  
  \begin{pmatrix}
    0\\-1
  \end{pmatrix}
  \text{\raisebox{-7pt}{, }}
  1+\dfrac{n(n+1)}{2}\Biggr)\text{\raisebox{-7pt}{.}}
\end{align*}
  Geometrically, $\Delta_n$ can be defined as follows (see Figure \ref{fig:arrangement}, \ref{Delta_3}). Let $\sigma_j$ be the square of side length $j$ for $j=1,\ldots,n$. First, we place $\sigma_1$ on $\mathbb{R}^2$ such that two of its sides lie on the positive $x-$axis and the negative $y-$axis. Next, suppose that $\sigma_{j-1}$ has been placed. Then we place $\sigma_j$ such that one of its diagonals lies on the line $y=-x$ and the intersection $\sigma_{j-1}\cap \sigma_j$ consists of a single point. We repeat this process up to $j=n$. Next, let $\sigma_n'$ be a copy of $\sigma_n$. We place $\sigma_n'$ at the position shifted from the placed $\sigma_n$ in the direction of the vector $\left(\begin{smallmatrix}1\\-1 \end{smallmatrix} \right)$. Then, the Delzant polytope $\Delta_n$ is the convex hull of the union of the placed squares $\sigma_1,\ldots,\sigma_n$, and $\sigma_n'$.

\begin{figure}
\centering 
\begin{tikzpicture}
    \draw[->, thick] (-0.5, 0) -- (7.5, 0) node[right] {$x$};
    \draw[->, thick] (0, -7.5) -- (0, 0.5) node[above] {$y$};
    \foreach \x in {0, ...,7} {
        \foreach \y in {-7, ...,0} {
        \pgfmathparse{\x==2 && \y==-2}
        \ifnum\pgfmathresult=0
            \fill[black!20] (\x,\y) circle (1.5pt);
            \fi
        }
    }
    \draw[thick, gray, fill=gray!50, opacity=0.7] 
        (0,0) -- (1,0) -- (1,-1) -- (0,-1) -- cycle;
        \node[scale=1.5] at (0.5, -0.5) {$\sigma_1$};
    \draw[thick, gray, fill=gray!50, opacity=0.7] 
        (1,-1) -- (3,-1) -- (3,-3) -- (1,-3) -- cycle;
        \node[scale=1.5] at (2, -2) {$\sigma_2$};
    \draw[thick, gray, fill=gray!50, opacity=0.7] 
        (3,-3) -- (6,-3) -- (6,-6) -- (3,-6) -- cycle;
        \node[scale=1.5] at (3.5, -3.5) {$\sigma_3$};
    \draw[thick, darkgray, fill=darkgray!50, opacity=0.7] 
        (4,-4) -- (7,-4) -- (7,-7) -- (4,-7) -- cycle;
        \node[scale=1.5] at (6.5, -6.5) {$\sigma_3'$};
\end{tikzpicture}
\caption{Arrangement of squares $(n=3)$}
\label{fig:arrangement}
\end{figure}
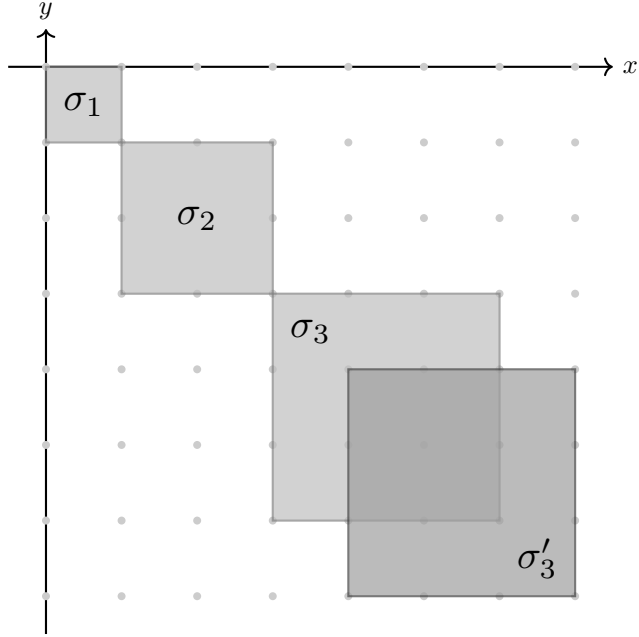

\begin{figure}
\centering 
\begin{tikzpicture}
    \draw[->, thick] (-0.5, 0) -- (7.5, 0) node[right] {$x$};
    \draw[->, thick] (0, -7.5) -- (0, 0.5) node[above] {$y$};
    \foreach \x in {0,1,2,3,4,5,6,7} {
        \foreach \y in {-7,-6,-5,-4,-3,-2,-1,0} {
            \fill[black!20] (\x,\y) circle (1.5pt);
        }
    }
    \draw[thick, blue, fill=blue!20, opacity=0.7] 
        (0,0) -- (1,0) -- (3,-1) -- (6,-3) -- (7,-4) -- (7,-7) -- (4,-7) -- (3,-6) -- (1,-3) -- (0,-1) -- cycle;
    \foreach \p in {(0,0),(1,0),(3,-1),(6,-3),(7,-4),(7,-7),(4,-7),(3,-6),(1,-3),(0,-1)} {
        \fill[red] \p circle (2pt);
    }
\end{tikzpicture}
\caption{The Delzant polytope $\Delta_3$. The red points are the vertices.}
\label{Delta_3}
\end{figure}
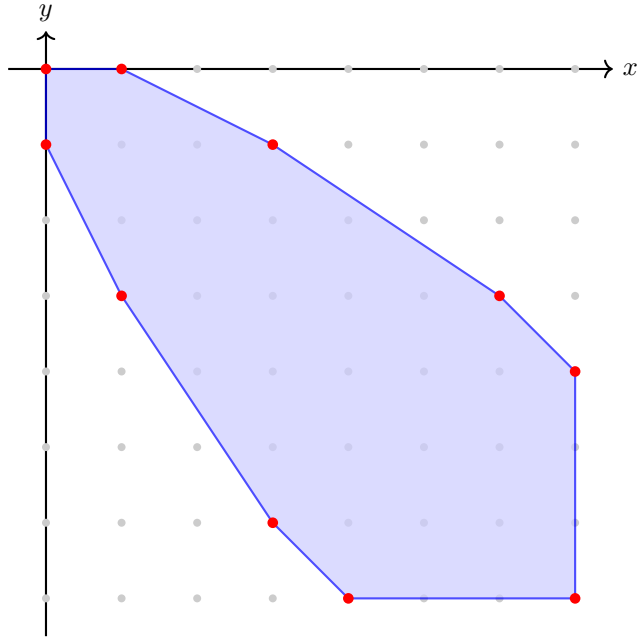

Next, we apply the Delzant construction to $\Delta_n$.
Let $\Pi:\mathbb{R}^{2n+4}\longrightarrow \mathbb{R}^2$ be the linear map defined by \vspace{2mm}
\begin{align*}
   \Pi(e_j)&= N_j^+\hspace{3mm}(1\leq j \leq n),\\
   \Pi(e_{n+j})&=N_j^-\hspace{3mm}(1\leq j \leq n), \\
   \Pi(e_{2n+1})&= N_s^+,\\
   \Pi(e_{2n+2})&= N_s^-,\\
   \Pi(e_{2n+3})&= N_x,\\
   \Pi(e_{2n+4})&=N_y\,.
\end{align*} 
Then, one can take $\{w_j^+, w_j^-\,(1\leq j\leq n), w_{2n+1}, w_{2n+2}\}$ as a basis for $\ker\Pi$, where the each vector is defined as follows:
\vspace{2mm}
\begin{align*}
   w_j^+ &= e_j+(j-1)e_{2n+2}+e_{2n+4} \hspace{3mm}(1\leq j\leq n), \\
   w_j^- &= e_{n+j}+(j-1)e_{2n+1}+e_{2n+3} \hspace{3mm} (1\leq j\leq n), \\
   w_{2n+1} &= e_{2n+1}+e_{2n+2}, \hspace{-23mm} & \\
   w_{2n+2} &= e_{2n+1}-e_{2n+3}+e_{2n+4}.
\end{align*}
The dual of the inclusion $\iota^*:(\mathbb{R}^{2n+4})^*\longrightarrow\mathfrak{k}^*$ has the following matrix representation with respect to the basis, \{$e_1^*, \ldots, e_{2n+4}^*$\} of $(\mathbb{R}^{2n+4})^*$ and $\{(w_j^+)^*_{1\leq j\leq n}, (w_j^-)^*_{1\leq j\leq n},\allowbreak (w_{2n+1})^*, (w_{2n+2})^*\}$ of $\mathfrak{k}^*$:\vspace{2mm}
\begin{equation*}
  \begin{matrix}
    \iota^* & = & 
    {\addtolength{\arraycolsep}{-3pt}
     \begin{pNiceMatrix}
    \Block{4-4}{{E^n}} &&&& \Vdots & \Block{4-4}{{O^n}} &&&& \Vdots & 0 & 0 & 0 & 1 \\
     &&&&     &&&&&     &0&1&0&1\\
     &&&&     &&&&&     & \vdots & \vdots & \vdots & \vdots \\
     &&&&     &&&&&     &0&n-1&0&1\\

     \myhdottedline \\
    
      \Block{4-4}{O^n} &&&&     & \Block{4-4}{E^n} &&&&     & 0 & 0 & 1 & 0 \\
     &&&&     &&&& &     & 1 & 0 & 1 & 0 \\
     &&&&     &&&& &     & \vdots & \vdots & \vdots & \vdots \\
     &&&&     &&&& &     & n-1 & 0 & 1 & 0 \\

    \myhdottedline \\

    \Block{2-4}{O^{2,n}} &&&&     & \Block{2-4}{O^{2,n}} &&&&     & 1 & 1 & 0 & 0 \\
     &&&&     &&&& &     & 1 & 0 & -1 & 1 
\end{pNiceMatrix}
}
  \end{matrix}\text{\raisebox{-75pt}{,}}
\end{equation*}
where $E^n$ is the $n$-by-$n$ identity matrix and $O^{j,k}$ is the $j$-by-$k$ zero matrix. (When $j=k$, we simply write $O^j$.) Therefore, the representation of the map $v=\iota^* \circ\,  v_0:\mathbb{C}^{2n+4}\longrightarrow\mathfrak{k}^*$ with respect to the above basis for $\mathfrak{k}^*$ is given by
\begin{align}\label{defofv}
  v(z_1, \ldots, z_{2n+4}) = & \sum_{j=1}^n\left(-\dfrac{1}{2}(|z_j|^2+(j-1)|z_{2n+2}|^2+|z_{2n+4}|^2)+c_j \right)(w_j^+)^* \\
   & + \sum_{j=1}^n\left(-\dfrac{1}{2}(|z_{n+j}|^2+(j-1)|z_{2n+1}|^2+|z_{2n+3}|^2)+c_j \right)(w_j^-)^* \\
   & + \left(-\dfrac{1}{2}(|z_{2n+1}|^2+|z_{2n+2}|^2)+2n\right)(w_{2n+1})^* \\
   & + \left(-\dfrac{1}{2}(|z_{2n+1}|^2-|z_{2n+3}|^2+|z_{2n+4}|^2)+n \right)(w_{2n+2})^*,
\end{align}
  where $c_j =\dfrac{1}{2}\left((2n+j)(j-1)+n(n+1)+2\right)$. Hence the subset $v^{-1}(0)\subset \mathbb{C}^{2n+4}$ is
determined by the following equations: 
\begin{align*}   
|z_j|^2+(j-1)|z_{2n+2}|^2+|z_{2n+4}|^2-2c_j &= 0 \hspace{3mm}(1\leq j\leq n),\\
|z_{n+j}|^2+(j-1)|z_{2n+1}|^2+|z_{2n+3}|^2-2c_j &= 0 \hspace{3mm}(1\leq j\leq n),\\
|z_{2n+1}|^2+|z_{2n+2}|^2-4n &= 0, \\
|z_{2n+1}|^2-|z_{2n+3}|^2+|z_{2n+4}|^2-2n &=0.
\end{align*}
Next, we describe the action of $K$ on $v^{-1}(0)$. We define a vector space isomorphism $\rho:\mathbb{R}^{2n+2}\longrightarrow\mathfrak{k}$ as follows:
\begin{align*}
\rho(e_j) &= w_j^+\hspace{2mm}(1\leq j\leq n),\\
\rho(e_{n+j}) &= w_j^-\hspace{2mm}(1\leq j\leq n),\\
\rho(e_{2n+1}) &= w_{2n+1},\\
\rho(e_{2n+2}) &= w_{2n+2}.
\end{align*}
Then, $\rho$ descends to a Lie group isomorphism $\widehat{\rho}:\mathbb{R}^{2n+2}/\mathbb{Z}^{2n+2}\longrightarrow K$. Therefore, $\widehat{\rho}$ induces the action of $\mathbb{R}^{2n+2}/\mathbb{Z}^{2n+2}$ on $v^{-1}(0)$.
Now, let $T_j\in \mathbb{R}^{2n+2}/\mathbb{Z}^{2n+2}$ be the element with $\frac{t_j}{2\pi}$ in the $j$-th entry and zeros elsewhere. Then, for $z=(z_1,\ldots,z_{2n+4})\in v^{-1}(0)$, the action of $T_j$ to $z$ is given as follows:
\begin{align*}
T_j\!\cdot\! z&=\bigl(z_1, \ldots, e^{it_j}z_j, \ldots, z_{2n}, z_{2n+1}, e^{i(j-1)t_j}z_{2n+2}, z_{2n+3}, e^{it_j}z_{2n+4} \bigr)\hspace{2mm}(1\leq j\leq n),\\
T_{n+j}\!\cdot\! z &=\bigl(z_1, \ldots, e^{it_{n+j}}z_{n+j}, \ldots, z_{2n}, e^{i(j-1)t_{n+j}}z_{2n+1}, z_{2n+2}, e^{it_{n+j}}z_{2n+3}, z_{2n+4} \bigr)\hspace{2mm}(1\leq j\leq n),\\
T_{2n+1}\!\cdot\! z &=\bigl(z_1, \ldots, z_{2n}, e^{it_{2n+1}}z_{2n+1}, e^{it_{2n+1}}z_{2n+2}, z_{2n+3}, z_{2n+4}\bigr),\\
T_{2n+2}\!\cdot\! z &=\bigl(z_1, \ldots, z_{2n}, e^{it_{2n+2}}z_{2n+1}, z_{2n+2}, e^{-it_{2n+2}}z_{2n+3}, e^{it_{2n+2}}z_{2n+4}\bigr).
\end{align*}
To obtain an expression for the moment map $\widehat{\mu}$, we determine $(\Pi^*)^{-1}$. The matrix representation of $\Pi$ with respect to the standard basis of $\mathbb{R}^{2n+4}$ and $\mathbb{R}^2$ is 
\begin{equation*}
     \Pi = 
     {\addtolength{\arraycolsep}{-3pt}
     \begin{pmatrix}
       0 & 1 & \ldots & n-1 & -1 & -2 & \ldots & -n & 1 & -1 & 1 & 0 \\
       1 & 2 & \ldots & n & 0 & -1 & \ldots & 1-n & 1 & -1 & 0 & -1
     \end{pmatrix}
     }\text{\raisebox{-8pt}{.}}
 \end{equation*}
Therefore, we obtain the following expression for $(\Pi^*)^{-1}$:
\vspace{5mm}
\begin{equation*}
     (\Pi^*)^{-1} =
     {\addtolength{\arraycolsep}{-3pt}
     \begin{pNiceMatrix}
       0 & \ldots & 0 & 1 & 0 \\
       0 & \ldots & 0 & 0 & -1
       \CodeAfter \OverBrace[yshift=3pt, shorten]{1-1}{1-3}{\scriptstyle 2n+2} 
     \end{pNiceMatrix}
     }\text{\raisebox{-8pt}{.}}
 \end{equation*}
Hence, the map $\mu = (\Pi^*)^{-1} \circ v_0 \circ i$ is given by 
\begin{equation*}
  \mu(z_1, \ldots, z_{2n+4}) = \left(-\dfrac{1}{2}|z_{2n+3}|^2+1+\dfrac{n(n+1)}{2}\text{\raisebox{-5pt}{,}}\hspace{1mm}\dfrac{1}{2}|z_{2n+4}|^2-1-\dfrac{n(n+1)}{2} \right)
  \text{\raisebox{-9pt}{.}}
\end{equation*}
In conclusion, the desired toric system $(M_n, \omega_n, \mu_n)$ is $(v^{-1}(0)/K, \omega_{\Delta_{n}}, \widehat{\mu})$.

\section{The Semitoric Systems \texorpdfstring{$(M_n,\omega_n, F_n)$}{}}\label{chapter4}
Let $\tilde{J}_n, f_n, \tilde{H}_n$ be the functions on $\mathbb{C}^{2n+4}$ defined by
\begin{align}
  \tilde{J}_n (z_1,  \ldots, z_{2n+4}) & = -\dfrac{1}{2}\left(|z_{2n+3}|^2-|z_{2n+4}|^2\right), \label{J_n}\\
  f_n (z_1,  \ldots, z_{2n+4}) & = \bar{z}_{2n+3}\bar{z}_{2n+4}\prod_{j=1}^{2n}z_j, \notag \\
  \tilde{H}_n (z_1,  \ldots, z_{2n+4}) & = \text{Re}\,f_n,\label{H_n}
\end{align}
respectively. Since the restrictions of $\tilde{J}_n$ and $\tilde{H}_n$ to $v^{-1}(0)$ are invariant under the $K$-action, they descend to the functions $J_n$ and $H_n$ on $M_n$. Moreover, since $\tilde{H}_n$ is invariant under the Hamiltonian $S^1$-action on $\mathbb{C}^{2n+4}$ generated by $\tilde{J}_n$, the same property holds for $H_n$ on $M_n$. Consequently, $J_n$ and $H_n$ are Poisson commutative.
Now, let $\tilde{F}_n:\mathbb{C}^{2n+4}\longrightarrow\mathbb{R}^2$ and $F_n:M_n\longrightarrow \mathbb{R}^2$ be the maps defined by $\tilde{F}_n =(\tilde{J}_n, \tilde{H}_n)$ and $F_n = (J_n, H_n)$, respectively.
\begin{theorem}\label{mainthm}
 For any natural number $n$, The tuple $(M_n, \omega_n, F_n)$ is a semitoric system with $2n$ focus-focus singularities.
\end{theorem}

To prove Theorem \ref{mainthm}, we determine the set of singularities of $F_n$ and examine the type of each singularity. 
Let $S_n$ be the set of pairs of natural numbers and $M^{k,l}$, $M^k$, and $M^{int}$ be the subsets of $M_n$ as follows:
\begin{align*}
  S_n = &\left\{\begin{pmatrix}1\\n+1\end{pmatrix}\text{\raisebox{-7pt}{,}} \begin{pmatrix}n\\2n+1\end{pmatrix}\text{\raisebox{-7pt}{,}} \begin{pmatrix}2n\\2n+2\end{pmatrix}\text{\raisebox{-7pt}{,}} \begin{pmatrix}2n+1\\2n+3\end{pmatrix}\text{\raisebox{-7pt}{,}} \begin{pmatrix}2n+2\\2n+4\end{pmatrix}\text{\raisebox{-7pt}{,}} \begin{pmatrix}2n+3\\2n+4\end{pmatrix}\right\}\\
        &\cup \bigcup_{j=1}^n\left\{\begin{pmatrix}j\\j+1\end{pmatrix}\text{\raisebox{-7pt}{,}} \begin{pmatrix}n+j\\n+j+1\end{pmatrix}\right\}\text{\raisebox{-7pt}{,}} 
\end{align*}
\begin{align*}
  M^{k,l} & = \left\{ [z_1,\ldots, z_{2n+4}]\in M_n \hspace{1mm}| \hspace{1mm} z_k =z_l =0, z_m\neq0 \hspace{2mm}(m\neq k,l)
  \right\}\hspace{3mm}\bigl(\text{$\bigl(\begin{smallmatrix}k\\l\end{smallmatrix}\bigr)$}\in S_n\bigr), \\
    M^k & =  \left\{ [z_1,\ldots, z_{2n+4}]\in M_n \hspace{1mm}| \hspace{1mm} z_k = 0, z_l \neq 0 \hspace{2mm} (k\neq l) \right\} \hspace{2mm} (1 \leq k \leq 2n+4), \\
     M^{int} & = \left\{ [z_1,\ldots, z_{2n+4}]\in M_n \hspace{1mm}| \hspace{1mm} z_k \neq 0 \hspace{2mm} (\text{for any} \,\,\,k=1,\ldots,2n+4\} \right\}\text{\raisebox{-2pt}{.}}
\end{align*}
Note that $M_n$ is the union of the subsets $M^{k,l}$, $M^k$, and $M^{\text{int}}$. Moreover, $M^{k,l}$ is a singleton. Therefore, to determine the singular set of $F_n$, it suffices to identify the singularities within each of the subsets of $M_n$ defined above.

In the remainder of this paper, we let $[p^{k,l}]\in M^{k,l}, [p^k]\in M^k,$ and $[p]\in M^{\text{int}}$. 

\begin{proposition}\label{prop:singtype}
The singularities of the map $F_n$ and their types are summarized in Table \ref{table:sing}.
  \begin{table}[h] 
  \centering     
  \caption{Singularities of $F_n$ and their types} 
  \label{table:sing}
  \renewcommand{\arraystretch}{1.1}
  \begin{tabular}{c|ccc}
    \hline
    Points & Rank of $dF_n$ & Types & Remarks\\ \hline 
 $[p^{1, n+1}]$& 0 & FF & - \\
 $[p^{j, j+1}]$& 0 & FF & $(1\leq j\leq n-1)$ \\
 $[p^{n+j, n+j+1}]$& 0 & FF & $(1\leq j\leq n-1)$\\
 $[p^{2n+3, 2n+4}]$& 0 & FF & - \\
 $[p^{n, 2n+1}]$& 1 & ER & $dJ_n=0$\\
 $[p^{2n, 2n+2}]$& 1 & ER & $dJ_n=0$\\
 $[p^{2n+1, 2n+3}]$& 1 & ER & $dJ_n=0$\\
 $[p^{2n+2, 2n+4}]$& 1 & ER & $dJ_n=0$ \\
 $[p^k]$& 2 & - & $(1\leq k\leq 2n)$ \\
 $[p^{2n+1}]$&$\leq 1$& EE or ER & $dJ_n=0$\\
 $[p^{2n+2}]$&$\leq 1$& EE or ER & $dJ_n=0$\\
 $[p^{2n+3}]$& 2 & - & - \\
 $[p^{2n+4}]$& 2 & - & - \\
 $[p]$&$\geq 1$& ER & - \\ \hline 
  \end{tabular}
  \renewcommand{\arraystretch}{1.0}
\end{table}
Moreover, for both $[p^k]$
 $(k=2n+1, 2n+2)$ and $[p]$, the rank conditions $\text{rank}\,(dF_n)_{[p^k]}=0$ and $\text{rank}\,(dF_n)_{[p]}=1$ are characterized by the same set of equations:
 \begin{equation*}
    \text{Im}\,f_n = 0,\, \text{and} \hspace{2mm} \dfrac{1}{|z_{2n+3}|^2}+\dfrac{1}{|z_{2n+4}|^2}-\sum_{j=1}^{2n}\dfrac{1}{|z_j|^2} = 0.
  \end{equation*}
\end{proposition}
Assuming Proposition \ref{prop:singtype}, we prove Theorem \ref{mainthm}.
\begin{proof}
[Proof of Thm \ref{mainthm}]
It remains to show that $F_n$ is regular on a dense subset of $M_n$. This follows from the fact that the singular set of $F_n$ is two-dimensional, which is a consequence of the non-degeneracy of all its singularities (see \cite[Proposition~3.16]{bolsinov-fomenko}).
\end{proof}

\section{Proof of Proposition \ref{prop:singtype}}\label{chapter5}

In the previous section, we assumed Proposition \ref{prop:singtype} to prove Theorem \ref{mainthm}. In this final section, we provide the proof of Proposition \ref{prop:singtype} and conclude the paper.
Henceforth, to simplify the notation, we denote the functions $\tilde{F}_n, \tilde{J}_n, f_n, \tilde{H}_n, F_n, J_n$, and $H_n$ by $\tilde{F}, \tilde{J}, f, \tilde{H}, F, J$, and $H$, respectively. In addition, Identifying $\mathbb{C}^{2n+4}$ and $\mathbb{R}^{4n+8}$ via the relation $z_j=x_j+iy_j$, we take as a basis of the tangent space at each point of $\mathbb{C}^{2n+4}$ as follows:
\begin{align*}
  \left( \pdv{}{x_1}\text{\raisebox{-10pt}{,}}\,\pdv{}{y_1}\text{\raisebox{-10pt}{,}}\,\text{\raisebox{-10pt}{\ldots}}\text{\raisebox{-10pt}{,}}\,\pdv{}{x_{2n+4}}\text{\raisebox{-10pt}{,}}\,\pdv{}{y_{2n+4}}\right)\raisebox{-10pt}{.}
\end{align*}
\vspace{3mm}
\subsection{Rank part}
First, we calculate the rank of the derivative of the map $F$. 
For a point $z\in v^{-1}(0)$, we define a subspace $(X_K)_z$ of $T_z(v^{-1}(0))$ as follows:
\begin{equation}
    (X_K)_z = \left\{\dfrac{d}{dt}\Big|_{t=0} \exp (tu)\!\cdot\! z\,\Bigg|\,u \in \mathfrak{k}\right\}\text{\raisebox{-10pt}{.}}
  \end{equation}
\begin{lemma}\label{lem:rank}
  For each $z\in v^{-1}(0)$, we have $\textnormal{rank}\,dF_{[z]}=\textnormal{rank}\,(d\tilde{F}|_{v^{-1}(0)})_{z}$.
\end{lemma}
\begin{proof}
  This follows from the fact that the map $\tilde{F}$ is invariant under the action of $K$ and that $\ker\,(dq)_z = (X_K)_z$ (Recall that $q$ is the quotient map in the diagram (\ref{diagram:Delzant})).
\end{proof}
Thus, to determine the rank of the map $F$ at $[z]\in M_n$, it suffices to calculate the rank of $\tilde{F}|_{v^{-1}(0)}$ at $z\in v^{-1}(0)$. To this end we describe $d\tilde{F}$ and the tangent space of $v^{-1}(0)$.

For any $j = 1,\ldots, 2n, 2n+3, 2n+4$, let $R_j, I_j:\mathbb{C}^{2n+4}\longrightarrow \mathbb{R}$ be the functions defined by 
\begin{align*}
   R_j = \left\{
    \begin{matrix}
    \text{Re}\left(\dfrac{f}{z_j}\right) & (j=1, \ldots, 2n), \\[3ex]
    \text{Re}\left(\dfrac{f}{\bar{z}_j}\right) & (j=2n+3, 2n+4),
  \end{matrix}\right. 
  & & \hspace{5mm}
  I_j = \left\{
    \begin{matrix}
    \text{Im}\left(\dfrac{f}{z_j}\right) & (j=1, \ldots, 2n), \\[3ex]
    \text{Im}\left(\dfrac{f}{\bar{z}_j}\right) & (j=2n+3, 2n+4).
    \end{matrix}\right.
& \hspace{5mm}
\end{align*}
\begin{lemma}
The Jacobian matrices of the functions $\tilde{J}$ and $\tilde{H}$ are given by
  \begin{align*}
    d\tilde{J}
    & = \bigl(0,\ldots,0,-x_{2n+3}, -y_{2n+3}, x_{2n+4}, y_{2n+4}\bigr), \\[3ex] 
    d\tilde{H}
    & = \bigl(R_1, -I_1, \ldots, R_{2n}, -I_{2n}, 0, 0, 0, 0, R_{2n+3}, I_{2n+3}, R_{2n+4}, I_{2n+4}\bigr)\text{\raisebox{-2pt}{.}}
  \end{align*}
\end{lemma}
\begin{lemma}
  At a point $z=p^{k,l}$, one can take a basis $\mathfrak{I}_{k,l}$ for $T_z(v^{-1}(0))$ as \begin{equation*}\mathfrak{I}_{k, l} = \{ \alpha^k, \beta^k, \alpha^l, \beta^l, w^1, w^2, \ldots, \widehat{w^k}, \ldots, \widehat{w^l}, \ldots, w^{2n+4} \} \subset T_z{\mathbb{C}^{2n+4}}.
\end{equation*}
Here, the notation $\widehat{*}$ indicates that the term $*$ is omitted, and letting $w\in \mathbb{R}^{4n+8}$ be $w = (w_m)_{1\leq m\leq 4n+8}$, we have
\begin{equation*}
  (\alpha^k)_m =\delta_{2k-1, m}, \hspace{3mm} (\alpha^l)_m =\delta_{2l-1, m}, \hspace{3mm} (\beta^k)_m =\delta_{2k, m}, \hspace{3mm} (\beta^l)_m =\delta_{2l, m}, \\
\end{equation*}
\begin{equation*}
  (w^j)_m = 
  \begin{cases}
    -y_j & (m=2j-1), \\
    x_j & (m=2j), \\
    0 & (otherwise).
  \end{cases}
\end{equation*}
In the case $z=p^k$, one can take a basis $\mathfrak{I}_k$ for $T_z(v^{-1}(0))$ as 
\begin{equation*}
\mathfrak{I}_k = \{ \epsilon^k, \alpha^k, \beta^k, w^1, w^2, \ldots, \widehat{w^k}, \ldots, w^{2n+4} \} \subset T_z{\mathbb{C}^{2n+4}},
\end{equation*}
where $\epsilon^k$ is an arbitrary vector satisfying the following:
\begin{itemize}
    \item For $k=1, \ldots, n$, we have
    \begin{gather*}
\bigl(dv_0(\epsilon^k)\bigr)_m=
      \begin{cases}
        m-k & (1\leq m\leq n),\\
        1-k-m+n & (n+1\leq m \leq2n),\\
        1 & (m=2n+1),\\
        -1 & (m=2n+2),\\
        k & (m=2n+3),\\
        k-1 & (m=2n+4).
      \end{cases}
    \end{gather*}
    
  \item For $k=n+1, \ldots, 2n$, we have
    \begin{gather*}
\bigl(dv_0(\epsilon^k)\bigr)_m=
     \begin{cases}
       m+k-1 & (1\leq m\leq n),\\
       k-m+n & (n+1\leq m \leq2n),\\
       1 & (m=2n+1),\\
       -1 & (m=2n+2),\\
       1-k & (m=2n+3),\\
       -k & (m=2n+4).
     \end{cases}
   \end{gather*}
   
\item For $k=2n+1, 2n+2$, we have
    \begin{gather*}
     \bigl(dv_0(\epsilon^k)\bigr)_m=
     \begin{cases}
       1 & (1\leq m\leq 2n),\\
       0 & (m=2n+1, 2n+2),\\
       -1 & (m=2n+3, 2n+4).
     \end{cases}
   \end{gather*}

   \item For $k=2n+3$, we have
    \begin{gather*}
     \bigl(dv_0(\epsilon^k)\bigr)_m=
     \begin{cases}
       m & (1\leq m\leq n),\\
       1-m+n & (n+1\leq m \leq2n),\\
       1 & (m=2n+1),\\
       -1 & (m=2n+2),\\
       0 & (m=2n+3),\\
       -1 & (m=2n+4).
     \end{cases}
   \end{gather*}

   \item For $k=2n+4$, we have
    \begin{gather*}
     \bigl(dv_0(\epsilon^k)\bigr)_m=
     \begin{cases}
       m-1 & (1\leq m\leq n),\\
       n-m & (n+1\leq m \leq2n),\\
       1 & (m=2n+1),\\
       -1 & (m=2n+2),\\
       1 & (m=2n+3),\\
       0 & (m=2n+4).
     \end{cases}
   \end{gather*}
\end{itemize}
In the case $z=p$, one can take a basis $\mathfrak{I}$ for $T_z(v^{-1}(0))$ as 
\begin{equation*}
\mathfrak{I} = \{ \epsilon^{2n+3}, \epsilon^{2n+4}, w^1, \ldots, w^{2n+4} \} \subset T_z{\mathbb{C}^{2n+4}}.
\end{equation*}
\end{lemma}
In the remainder of this paper, we choose $\epsilon^k$ in the above lemma as follows. Suppose now that $(dv_0(\epsilon^k))_m=r$. Then,
\begin{gather*}
  (\epsilon^k)_{2m-1}=-\dfrac{rx_m}{|z_m|^2},\hspace{5mm}(\epsilon^k)_{2m}=-\dfrac{ry_m}{|z_m|^2}.
\end{gather*}
Having established the basis for all cases, we now compute the rank of $dF$.
\begin{proof}[Proof of Proposition \ref{prop:singtype} (Rank Part)]
  At the point $[z]=[p^{k,l}]$, 
\begin{itemize}
  \item if neither $k$ nor $l$ is $2n+1$ or $2n+2$,
  \begin{equation*} d\tilde{F}(\mathfrak{I}_{k, l}) = 
    \begin{pmatrix}
      d\tilde{J} \\
      d\tilde{H}
    \end{pmatrix}
    (\mathfrak{I}_{k, l}) = 0.
  \end{equation*}
  Therefore, by Lemma \ref{lem:rank}, we have $\textnormal{rank}\,dF_{[z]}=\textnormal{rank}\,(d\tilde{F}|_{v^{-1}(0)})_{z} = 0$.
\vspace{3mm}
  \item If $k=2n+1$ or $2n+2$, (therefore $l=$)
  \begin{equation*}  d\tilde{F}(\mathfrak{I}_{k, l}) =
    {\addtolength{\arraycolsep}{-3pt}
    \begin{pmatrix}
      0 & 0 & 0 & 0 & 0 & \ldots & 0 \\
      0 & 0 & R_l & I_l & 0 & \ldots & 0
    \end{pmatrix}
    }\text{\raisebox{-7pt}{.}}
    \end{equation*}
Therefore, we have $dJ_{[z]}=0$ and also have $\textnormal{rank}\,dF_{[z]} = 1$ since at least one of $R_l$ and $I_l$ is non-zero.
\vspace{3mm}
\item If $l=2n+1$ or $2n+2$,
\begin{equation*}
d\tilde{F}(\mathfrak{I}_{k, l}) =
    {\addtolength{\arraycolsep}{-3pt}
    \begin{pmatrix}
      0 & 0 & 0 & \ldots & 0 \\
      R_k & -I_k & 0 & \ldots & 0
    \end{pmatrix}
    }\text{\raisebox{-7pt}{.}}
    \end{equation*}
As in the previous case, we have $dJ_{[z]}=0$ and $\textnormal{rank}\,dF_{[z]} = 1$.
\end{itemize}
\vspace{3mm}

In the case $[z]=[p^k]$, 
\vspace{3mm}
\begin{itemize}
  \item if $k=1,\ldots,2n$, we have 
   \begin{equation*} d\tilde{F}(\mathfrak{I}_k) =
    {\addtolength{\arraycolsep}{-3pt}
    \begin{pNiceMatrix}
      1&0&0&*&\ldots&*\\
      *&R_k&-I_k&*&\ldots&*
    \end{pNiceMatrix}
    }\text{\raisebox{-7pt}{.}}
    \end{equation*}
Since at least one of $R_k$ and $I_k$ is non-zero, we have $\textnormal{rank}\,dF_{[z]} = 2$.
\vspace{3mm}
\item If $k=2n+1$ or $2n+2$, we have 
\begin{equation*}
d\tilde{F}(\mathfrak{I}_k) =
    {\addtolength{\arraycolsep}{-3pt}
    \begin{pmatrix}
      0 & 0 & 0 & 0 & 0 & \ldots & 0\\
     \left(\dfrac{1}{|z_{2n+3}|^2}+\dfrac{1}{|z_{2n+4}|^2}-\sum\limits_{j=1}^{2n} \dfrac{1}{|z_j|^2}\right)\tilde{H} & 0 & 0 & 0 & \text{Im}f & \ldots & \text{Im}f
    \end{pmatrix}
    }\text{\raisebox{-14pt}{.}}
    \end{equation*}
Therefore, we have $dJ_{[z]}=0$. In addition,  $\text{rank}\,dF_{[z]}=0$ 
holds if and only if 
\begin{equation*}
    \text{Im}\,f = 0 \hspace{2mm}\text{and} \hspace{3mm} \dfrac{1}{|z_{2n+3}|^2}+\dfrac{1}{|z_{2n+4}|^2}-\sum\limits_{j=1}^{2n} \dfrac{1}{|z_j|^2} = 0
\end{equation*}
  holds.
\vspace{3mm}
\item If $k=2n+3, 2n+4$, we have 
\begin{equation*}
d\tilde{F}(\mathfrak{I}_k) =
    {\addtolength{\arraycolsep}{-3pt}
    \begin{pmatrix}
      1&0&0&*&\ldots&*\\
      *&R_k&I_k&*&\ldots&*
    \end{pmatrix}
    }\text{\raisebox{-7pt}{.}}
    \end{equation*}
Since at least one of $R_k$ and $I_k$ is non-zero, we have $\text{rank}\,dF_{[z]}=2$.
\end{itemize}
\vspace{3mm}

In the case $[z]=[p]$, we have 
\begin{equation*}
    d\tilde{F}(\mathfrak{I}) =
    {\addtolength{\arraycolsep}{-3pt}
    \begin{pmatrix}
      1&1&0&0&0&\ldots&0&0&0\\
      \kappa_1&\kappa_2&0&0& 
      -\text{Im}f & \ldots & -\text{Im}f & \text{Im}f & \text{Im}f
    \end{pmatrix}
    }\text{\raisebox{-6pt}{,}}
    \end{equation*}
where 
\begin{align*}
      \kappa_1 &= \left(\dfrac{1}{|z_{2n+4}|^2}-\sum_{j=1}^n\left(\dfrac{j}{|z_j|^2}+\dfrac{1-j}{|z_{n+j}|^2}\right)\right)\tilde{H}, \\
      \kappa_2 &= \left(-\dfrac{1}{|z_{2n+3}|^2}+\sum_{j=1}^n\left(\dfrac{1-j}{|z_j|^2}+\dfrac{j}{|z_{n+j}|^2}\right)\right)\tilde{H}.
    \end{align*}
Therefore, $\text{rank}\,dF_{[z]}=1$ holds if and only if 
\begin{equation*}
    \text{Im}\,f = 0 \hspace{2mm}\text{and} \hspace{3mm} \dfrac{1}{|z_{2n+3}|^2}+\dfrac{1}{|z_{2n+4}|^2}-\sum\limits_{j=1}^{2n} \dfrac{1}{|z_j|^2} = 0
  \end{equation*}
  holds.
\end{proof}
\vspace{3mm}
\subsection{The types of the singularities}

Finally, we determine the types of the singularities of $F$. Before that, recall that the symbol $O^{j,k}$ denotes the $j$-by-$k$ zero matrix. (When $j=k$, we simply write $O^j$.) 

First, we recall the method of Lagrange multipliers. 
Consider a function $h:\mathbb{R}^m\longrightarrow\mathbb{R}$ and a map $v=(v_1,\ldots, v_l):\mathbb{R}^m\longrightarrow\mathbb{R}^l$. Let $0\in \mathbb{R}^l$ be a regular value of $v$, and let $p=(\bm{\xi}(p), \bm{\eta}(p)) = \bigl(\xi_1(p), \xi_2(p), \ldots, \xi_{m-l}(p),\allowbreak \eta_{1}(p),  \eta_{2}(p),\allowbreak \ldots, \eta_{l}(p)\bigr)\in v^{-1}(0)$ be a critical point of the restriction $h|_{v^{-1}(0)}$. Regarding the Jacobian matrix of $v$ at $p$, denoted by $dv_p = \Bigl(\,\left(\pdv{v}{\xi}\right)\text{\raisebox{-2pt}{\vdots}} \,\left(\pdv{v}{\eta}\right)\Bigr)$, we assume that the submatrix $\left(\pdv{v}{\eta}\right)$ is non-singular. Then, by the implicit function theorem, there exists a map $\varphi$ from an open neighborhood of $\bm{\xi}(p)\in\mathbb{R}^{m-l}$ to an open neighborhood of $\bm{\eta}(p)\in\mathbb{R}^l$ such that the map $\Phi(\bm{\xi})=(\bm{\xi},\varphi(\bm{\xi}))$ provides a chart for $v^{-1}(0)$ around $p$. Furthermore, the relation
$d\varphi = -\left(\pdv{v}{\eta}\right)^{-1}\Bigl(\pdv{v}{\xi}\Bigr)$ holds at the point $\bm{\xi}(p)$.

\begin{lemma}\label{lem:laghessian}
Define a function $f_0:\mathbb{R}^{m+l}\longrightarrow\mathbb{R}$ with Lagrange multipliers $\bm{\lambda}=(\lambda_1, \ldots, \lambda_l)$ by 
\begin{equation*}
f_0 = -\sum_{j=1}^l\lambda_jv_j.
\end{equation*}
Furthermore, for the function $F_h = h+f_0$, let $d^2F_h(p,\bm{\lambda})$ be the Hessian of $F_h$ at a point $(p,\bm{\lambda})$ satisfying the equation $dF_h(p,\bm{\lambda})=0$, denoted by 
\begin{equation*}
  d^2F_h(p, \bm{\lambda}) =
  \begin{pNiceMatrix}
    \Block{1-4}{\huge{(F_h)_{\xi\xi}}}&&&&\Vdots&\Block{1-4}{\huge{(F_h)_{\xi\eta}}}&&&&\Vdots&\Block{1-4}{\huge{(F_h)_{\xi\lambda}}}&&& \\
    \myhdottedline\\[-10pt]
    \Block{1-4}{\huge{(F_h)_{\eta\xi}}}&&&& &\Block{1-4}{\huge{(F_h)_{\eta\eta}}}&&&& &\Block{1-4}{\huge{(F_h)_{\eta\lambda}}}&&& \\
    \myhdottedline\\[-10pt]
    \Block{1-4}{\huge{(F_h)_{\lambda \xi}}}&&&& &\Block{1-4}{\huge{(F_h)_{\lambda \eta}}}&&&& &\Block{1-4}{\huge{(F_h)_{\lambda \lambda}}}&&& 
  \end{pNiceMatrix}
  \text{\raisebox{-19pt}{.}}
\end{equation*}
Then, the Hessian $d^2(h\circ \Phi)(\bm{\xi}(p))$ of the function $h\circ \Phi$ at the point $\bm{\xi}(p)$ satisfies
\begin{align*}
  d^2(h\circ \Phi)(\bm{\xi}(p)) = (F_h)_{\xi\xi}+(F_h)_{\xi\eta}d\varphi_{\bm{\xi}(p)}+{}^t\!((F_h)_{\xi\eta}d\varphi_{\bm{\xi}(p)})+{}^t\!(d\varphi_{\bm{\xi}
(p)})(F_h)_{\eta\eta}(d\varphi_{\bm{\xi}(p)})\text{\raisebox{-2pt}{.}}
\end{align*}
\end{lemma}

Let $[z]\in M_n$ be a singularity of the map $F$. The information required to determine the type of the singularity is summarized in the following lemma.
For any $j,k =1,\ldots, 2n, 2n+3, 2n+4$, let $R_{j,k}, I_{j,k}:\mathbb{C}^{2n+4}\longrightarrow \mathbb{R}$ be the functions defined by
\begin{align*}
   R_{j,k} &= \left\{
    \begin{matrix}
    \text{Re}\left(\dfrac{f}{z_j z_k}\right) & (j,k=1,\ldots,2n), \\[3ex]
    \text{Re}\left(\dfrac{f}{z_j \bar{z}_k}\right) & 
    {\begin{pmatrix}j=1,\ldots,2n \\ k=2n+3, 2n+4
    \end{pmatrix}}
    \text{\raisebox{-7pt}{,}}\\[3ex]
    \text{Re}\left(\dfrac{f}{\bar{z}_j\bar{z}_k}\right) & (j, k=2n+3, 2n+4),
    \end{matrix}\right. 
\\[3ex]
I_{j,k} &= \left\{
    \begin{matrix}
    \text{Im}\left(\dfrac{f}{z_j z_k}\right) & (j,k=1,\ldots,2n), \\[3ex]
    \text{Im}\left(\dfrac{f}{z_j \bar{z}_k}\right) & {\begin{pmatrix}j=1,\ldots,2n \\ k=2n+3, 2n+4
    \end{pmatrix}}
    \text{\raisebox{-7pt}{,}}\\[3ex]
    \text{Im}\left(\dfrac{f}{\bar{z}_j\bar{z}_k}\right) & (j, k=2n+3, 2n+4).
    \end{matrix} \right.
\end{align*}
\vspace{3mm}
\begin{lemma}
The Jacobian matrix of the map $v$ defined by (\ref{defofv}) and the Hessians of the functions $\tilde{J}$ and $\tilde{H}$ are given as follows:
\end{lemma}
\begin{gather*}
dv=
     {\addtolength{\arraycolsep}{-2pt}
  \begin{pNiceMatrix}
  \Block{4-6}{V}&&&&&&\Block{2-2}{O^{n,2}}&&\Block{2-2}{V_{2n+2}}&&\Block{2-2}{O^{n,2}}&&\Block{2-2}{W_{2n+4}}&\\
  &&&&&&&&&&&&&\\
  &&&&&&\Block{2-2}{V_{2n+1}}&&\Block{2-2}{O^{n,2}}&&\Block{2-2}{W_{2n+3}}&&\Block{2-2}{O^{n,2}}&\\
  &&&&&&&&&&&&&\\
  \Block{2-6}{O^{2,4n}}&&&&&&-x_{2n+1}&-y_{2n+1}&-x_{2n+2}&-y_{2n+2}&&&&\\
  &&&&&&-x_{2n+1}&-y_{2n+1}&&&x_{2n+3}&y_{2n+3}&-x_{2n+4}&-y_{2n+4}
   \CodeAfter
    \tikz \draw [dotted] (1-|7) -- (7-|7);
    \tikz \draw [dotted] (1-|9) -- (5-|9);
    \tikz \draw [dotted] (1-|11) -- (5-|11);
    \tikz \draw [dotted] (1-|13) -- (5-|13);
    \tikz \draw [dotted] (3-|7) -- (3-|15);
    \tikz \draw [dotted] (5-|1) -- (5-|15);
    \end{pNiceMatrix}\text{\raisebox{-31pt}{,}}
 }\\[3ex]
  d^2\tilde{J}=
  {\addtolength{\arraycolsep}{-2pt}
  \begin{pNiceMatrix}
    \Block{2-6}{O^{4n+4}}&&&&&&\Vdots&&&&\\
    &&&&&& &&&&\\
    \myhdottedline\\[-10pt]
    &&&&&& &-1&&&\\
    &&&&&& &&-1&&\\
  &&&&&& &&&1&\\
  &&&&&& &&&&1
  \end{pNiceMatrix}
  }\text{\raisebox{-35pt}{,}}
  \hspace{5mm}
  d^2\tilde{H} = 
    \begin{pNiceMatrix}
    \Block{2-4}{\raisebox{1pt}{$A_{\tilde{H}}$}} &&&&\Vdots& \Block{2-2}{O^{4n,4}} &&\Vdots& \Block{2-2}{\huge{B_{\tilde{H}}}} & \\
     &&&&&&&& \\
     \myhdottedline\\[-10pt]
    \Block{2-4}{\raisebox{10pt}{$O^{4,4n}$}} &&&&& \Block{2-2}{\raisebox{10pt}{$O^4$}} &&& \Block{2-2}{\raisebox{10pt}{$O^4$}} & \\
     &&&&&&&&& \\
     \myhdottedline\\[-10pt]
    \Block{2-4}{\raisebox{12pt}{${}^t\!B_{\tilde{H}}$}} &&&&& \Block{2-2}{\raisebox{10pt}{$O^4$}} &&& \Block{2-2}{\raisebox{10pt}{$C_{\tilde{H}}$}} & \\
     &&&&&&&&& \\
    \end{pNiceMatrix}
    \text{\raisebox{-35pt}{,}}
\end{gather*}
where
\begin{align*}
V&=
{\addtolength{\arraycolsep}{-3pt}
\begin{pmatrix}
  -x_1&-y_1&&&\\
  &&\ddots&&\\
  &&&-x_{2n}&-y_{2n}
\end{pmatrix}\text{\raisebox{-15pt}{,}}
}\hspace{4mm}
V_k=
{\addtolength{\arraycolsep}{-3pt}
\begin{pmatrix}
  0&0\\
  -x_k&-y_k\\
  \vdots&\vdots\\
  (1-j)x_k&(1-j)y_{k}\\
  \vdots&\vdots\\
  (1-n)x_{k}&(1-n)y_{k}
\end{pmatrix}\text{\raisebox{-35pt}{,}}
}\\
W_k&=
{\addtolength{\arraycolsep}{-3pt}
\begin{pmatrix}
  -x_{k}&-y_{k}\\
  \vdots&\vdots\\
  -x_{k}&-y_{k}\\
\end{pmatrix}\text{\raisebox{-15pt}{,}}
}\\[3ex]
A_{\tilde{H}} &= 
  {\addtolength{\arraycolsep}{-3pt}
    \begin{pNiceMatrix}
    0&0&R_{1,2}&-I_{1,2}&R_{1,3}&-I_{1,3}&\ldots&R_{1,2n}&-I_{1,2n}\\
    0&0&-I_{1,2}&-R_{1,2}&-I_{1,3}&-R_{1,3}&\ldots&-I_{1,2n}&-R_{1,2n}\\
    R_{1,2}&-I_{1,2}&0&0&R_{2,3}&-I_{2,3}&\ldots&R_{2,2n}&-I_{2,2n}\\
    -I_{1,2}&-R_{1,2}&0&0&-I_{2,3}&-R_{2,3}&\ldots&-I_{2,2n}&-R_{2,2n}\\
    R_{1,3}&-I_{1,3}&R_{2,3}&-I_{2,3}&0&0&\ldots&R_{3,2n}&-I_{3,2n}\\
    -I_{1,3}&-R_{1,3}&-I_{2,3}&-R_{2,3}&0&0&\ldots&-I_{3,2n}&-R_{3,2n}\\
    \vdots&\vdots&\vdots&\vdots&\vdots&\vdots&\ddots&\vdots&\vdots\\
    R_{1,2n}&-I_{1,2n}&R_{2,2n}&-I_{2,2n}&R_{3,2n}&-I_{3,2n}&\ldots&0&0\\
    -I_{1,2n}&-R_{1,2n}&-I_{2,2n}&-R_{2,2n}&-I_{3,2n}&-R_{3,2n}&\ldots&0&0 
    \end{pNiceMatrix}
    }\text{\raisebox{-52pt}{,}}\\[3ex]
    B_{\tilde{H}}&=
  {\addtolength{\arraycolsep}{-3pt}
    \begin{pmatrix}
      R_{1, 2n+3} & I_{1, 2n+3} & R_{1, 2n+4} & I_{1, 2n+4} \\
      -I_{1, 2n+3} & R_{1, 2n+3} & -I_{1, 2n+4} & R_{1, 2n+4} \\
      R_{2, 2n+3} & I_{2, 2n+3} & R_{2, 2n+4} & I_{2, 2n+4} \\
      -I_{2, 2n+3} & R_{2, 2n+3} & -I_{2, 2n+4} & R_{2, 2n+4} \\
      \vdots&\vdots&\vdots&\vdots\\
      R_{2n, 2n+3} & I_{2n, 2n+3} & R_{2n, 2n+4} & I_{2n, 2n+4} \\
      -I_{2n, 2n+3} & R_{2n, 2n+3} & -I_{2n, 2n+4} & R_{2n, 2n+4} 
    \end{pmatrix}
    }\text{\raisebox{-40pt}{,}}\\[3ex]
  C_{\tilde{H}}&= 
    {\addtolength{\arraycolsep}{-3pt}
    \begin{pNiceMatrix}
     0&0&\Vdots&R_{2n+3,2n+4}&I_{2n+3,2n+4}\\
     0&0& &I_{2n+3,2n+4}&-R_{2n+3,2n+4}\\
     \myhdottedline\\[-10pt]
     R_{2n+3,2n+4}&I_{2n+3,2n+4}& &0&0\\ 
     I_{2n+3,2n+4}&-R_{2n+3,2n+4}& &0&0 
    \end{pNiceMatrix}
    }\text{\raisebox{-20pt}{.}}\\
    \end{align*}

We express the map $v$ in components as $v=(v_1,\ldots,v_{2n+2})$. Let $f_0:\mathbb{C}^{2n+4}\longrightarrow\mathbb{R}$ be the function defined by 
\begin{equation*}
f_0 = -\sum_{j=1}^{2n+2}\lambda_jv_j,
\end{equation*}
where $\lambda_1,\ldots,\lambda_{2n+2}$ are Lagrange multipliers. We also set $F_{\tilde{J}}=\tilde{J}+f_0$ and $F_{\tilde{H}}=\tilde{H}+f_0$.

When a point $[z]$ is a singularity of rank\,$0$, the procedure for determining the type of the singularity is as follows:

\begin{enumerate}[label=Step\,\arabic*. ]
  \item Solve the equations $dF_{\tilde{J}}(z, \bm{\lambda}) = 0$ and $dF_{\tilde{H}}(z, \bm{\lambda}') = 0$ for $\bm{\lambda}$ and $\bm{\lambda}'$, respectively.
  \item Using Lemma \ref{lem:laghessian}, calculate the representation matrices of $d^2(\tilde{J}\circ\Phi)_{(\Phi^{-1}(z))}$ and $d^2(\tilde{H}\circ\Phi)_{(\Phi^{-1}(z))}$ based on $d^2F_{\tilde{J}}(z)$, $d^2F_{\tilde{H}}(z)$, respectively, and $d\varphi_{(\Phi^{-1}(z))}$. In addition, calculate the representation matrix of the pullback $(\Phi^*\omega_{\text{std}})_{\Phi^{-1}(z)}$ of the standard symplectic form $\omega_{\text{std}}$ on $\mathbb{C}^{2n+4}$ at $z$ via $\Phi$.
  \item Let $\{q_1, \ldots, q_{2n+6}\}$ be a basis of $T_{\Phi^{-1}(z)}\mathbb{R}^{2n+6}$ obtained by extending a basis $\{q_1, \ldots, q_{2n+2}\}$ of $d\Phi^{-1}(X_K)_z$. The $4$-by-$4$ submatrices, consisting of components from $(2n+3,2n+3)$ to $(2n+6,2n+6)$, of the matrices obtained by pulling back $d^2(\tilde{J}\circ\Phi)_{(\Phi^{-1}(z))}$, $d^2(\tilde{H}\circ\Phi)_{(\Phi^{-1}(z))}$, and $(\Phi^*\omega_{\text{std}})_{\Phi^{-1}(z)}$ via the basis $\{q_1,\ldots,q_{2n+6}\}$ are the representation matrices of the bilinear forms $d^2J_{[z]}$, $d^2H_{[z]}$, and $\omega_{[z]}$ on $\mathbb{R}^{4}\cong T_{[z]}v^{-1}(0)/K$, respectively. We denote these matrices by $d^2J$, $d^2H$, and $\Omega$, respectively.
  \item If $d^2J$ and $d^2H$ are linearly dependent, the singularity $[z]$ is degenerate. Otherwise, proceed to Step 5.
  \item If $a_1d^2J+a_2d^2H$ is a singular matrix for any $a_1, a_2\in \mathbb{R}$, then $[z]$ is degenerate. Otherwise, 
choose real numbers $a_1, a_2$ such that $a_1d^2J+a_2d^2H$ is non-singular and solve the characteristic equation (\ref{chreq}).
  \item Determine the type of the singularity $[z]$ based on (\ref{chtype1}) and (\ref{chtype2}).
\end{enumerate}
\vspace{3mm}

When a point $[z]$ is a singularity of rank\,$1$, the procedure is as follows:

\begin{enumerate}[label=Step\,\arabic*.]
  \item Find $(b_1, b_2) \in \mathbb{R}^2 \setminus \{0\}$ satisfying $b_1dJ_{[z]} + b_2dH_{[z]} = 0$. Then, the function $g = b_1J + b_2H$ satisfies $dg = 0$ at the point $[z]$. We then set $\tilde{g} = b_1\tilde{J} + b_2\tilde{H}$ and $F_{\tilde{g}} = \tilde{g} + f_0$.
  \item Solve the equation $dF_{\tilde{g}}(z, \bm{\lambda}) = 0$ for $\bm{\lambda}$.
  \item Using Lemma \ref{lem:laghessian}, calculate the representation matrix of $d^2(\tilde{g}\circ \Phi)_{\Phi^{-1}(z)}$ based on $d^2F_{\tilde{g}}(z)$ and $d\varphi_{\Phi^{-1}(z)}$. Additionally, calculate the representation matrix of the pullback $(\Phi^*\omega_{\text{std}})_{\Phi^{-1}(z)}$ of the standard symplectic form $\omega_{\text{std}}$ on $\mathbb{C}^{2n+4}$ at $z$ via $\Phi$.
  \item Suppose that $L_{[z]} \subset T_{[z]}M_n$ is spanned by $X_J([z])$. (Namely, we assume $dJ_{[z]} \neq 0$. If $dJ_{[z]} = 0$, then $dH_{[z]} \neq 0$ must hold, and we take $X_H([z])$ as the spanning vector of $L_{[z]}$.)
  Let $\{q_1, \dots, q_{2n+2}, d\Phi^{-1}(X_{\tilde{J}})(z), q_{2n+4}, q_{2n+5},q_{2n+6}\}$ be a basis of $T_{\Phi^{-1}(z)}\mathbb{R}^{2n+6}$ obtained by extending a basis $\{q_1, \ldots, q_{2n+2}\}$ of $d\Phi^{-1}(X_K)_z$. We pull back $d^2(\tilde{g}\circ\Phi)_{(\Phi^{-1}(z))}$ and $(\Phi^*\omega_{\text{std}})_{\Phi^{-1}(z)}$ via this basis. Then, the $4$-by-$4$ submatrices, consisting of the entries from $(2n+3, 2n+3)$ to $(2n+6, 2n+6)$ of the resulting representation matrices, are the representation matrices of the bilinear forms $d^2g_{[z]}$ and $\omega_{[z]}$ on $\mathbb{R}^4 \cong T_{[z]}v^{-1}(0)/K$. We denote these matrices by $d^2G$ and $\Omega$, respectively.
  \item Let $\{e_1, e_2, e_3, e_4\}$ be the standard basis of $\mathbb{R}^4$. Since $L$ is spanned by $e_1$, we can explicitly determine $L^\omega$ using $\Omega$. Based on this, calculate the representation matrices $\widehat{d^2G}$ and $\widehat{\Omega}$ of the induced forms $\widehat{d^2g}_{[z]}$ and $\widehat{\omega}_{[z]}$ on $L^{\omega}/L$, respectively. Using these, solve the characteristic equation (\ref{chreq}).
  \item Determine the type of the singularity $[z]$ by referring to (\ref{chtype1}) and (\ref{chtype2}).
\end{enumerate}

First, we determine the type of the singularities $[z]=[p^{k,l}]$.

\subsubsection{$[z]=[p^{k,l}]$ $\mathrm{(}$The case where neither $k$ nor $l$ is $2n+1$ or $2n+2 \mathrm{)}$}\mbox{}\\

Recall that each point $[z]$ is a singularity of rank\,$0$. The essential part of the computation for the singularity type is fully captured by the case $k=2n+3$ and $l=2n+4$. In what follows, we describe the calculation process for this specific case in detail.

A representative of the point $[z]$ can be taken as $z=(x_1, \ldots, x_{2n+2},0,0)\in v^{-1}(0)$ with $x_j\in \mathbb{R}$. That is, we may assume that all components of the representative $z$ are real.

\begin{enumerate}[label=Step\,\arabic*. ]
  \item Solving the equations $dF_{\tilde{J}}(z, \bm{\lambda}) = 0$ and $dF_{\tilde{H}}(z, \bm{\lambda'}) = 0$ for $\bm{\lambda}$ and $\bm{\lambda'}$, we obtain $\bm{\lambda} = \bm{\lambda'} = 0$.

  \item The $d^2F_{\tilde{J}}(z)$ and $d^2F_{\tilde{H}}(z)$ are given by: 
  \vspace{3mm}
  \begin{gather*}
    d^2F_{\tilde{J}}(z) = 
    {\addtolength{\arraycolsep}{-3pt}
    \begin{pNiceMatrix}
    \Block{2-7}{O^{4n+4}}&&&&&&&\Vdots&&&& \\
    &&&&&&&&&&&\\
    \myhdottedline\\[-10pt]
    &&&&&&& &-1&&& \\
    &&&&&&& &&-1&& \\
    &&&&&&& &&&1& \\
    &&&&&&& &&&&1 \\
    \end{pNiceMatrix}
    }
    \text{\raisebox{-34pt}{,}}
    \hspace{5mm}
   d^2F_{\tilde{H}}(z) = 
    {\addtolength{\arraycolsep}{-3pt}
    \begin{pNiceMatrix}
    \Block{2-7}{O^{4n+4}}&&&&&&&\Vdots&&&& \\
    &&&&&&&&&&&\\
    \myhdottedline\\[-10pt]
    &&&&&&& &&&c&0 \\
    &&&&&&& &&&0&-c \\
    &&&&&&& &c&0&& \\
    &&&&&&& &0&-c&& \\
    \end{pNiceMatrix}
    }\text{\raisebox{-34pt}{,}}
  \end{gather*}
where $c=R_{2n+3,2n+4}$. Next, we compute $d\varphi_{\Phi^{-1}(z)}$. The matrix $\Bigl(\pdv{v}{\eta}\Bigr)$ consisting of the non-zero column vectors of $dv$ at $z$ is non-singular. Therefore, by the implicit function theorem, there exists a map $\varphi=(\varphi_1, \ldots, \varphi_{2n+2})$ from an open neighborhood of $0\in \mathbb{R}^{2n+6}$ to an open neighborhood of $(x_1, \ldots, x_{2n+2})\in \mathbb{R}^{2n+2}$ such that the map $\Phi = (\varphi_1, \xi_1, \ldots, \varphi_{2n+2}, \xi_{2n+2}, \allowbreak \xi_{2n+3}, \xi_{2n+4}, \xi_{2n+5}, \xi_{2n+6})$ provides a chart for $v^{-1}(0)$ around the point $z$. Moreover, in this case, since $\Bigl(\pdv{v}{\xi}\Bigr)=0$, we see that $d\varphi_{\Phi^{-1}(z)}=0$. 
Thus, by Lemma \ref{lem:laghessian}, we obtain 
\begin{align*}
    d^2(\tilde{J}\circ\Phi)_{(\Phi^{-1}(z))} &=
    {\addtolength{\arraycolsep}{-3pt}
    \begin{pNiceMatrix}
    \Block{2-7}{O^{2n+2}}&&&&&&&\Vdots&&&& \\
    &&&&&&&&&&&\\
    \myhdottedline\\[-10pt]
    &&&&&&& &-1&&& \\
    &&&&&&& &&-1&& \\
    &&&&&&& &&&1& \\
    &&&&&&& &&&&1 \\
    \end{pNiceMatrix}
    }
    \text{\raisebox{-34pt}{, }}\\[3ex]
   d^2(\tilde{H}\circ\Phi)_{(\Phi^{-1}(z))} &= 
    {\addtolength{\arraycolsep}{-3pt}
    \begin{pNiceMatrix}
    \Block{2-7}{O^{2n+2}}&&&&&&&\Vdots&&&& \\
    &&&&&&&&&&&&\\
    \myhdottedline\\[-10pt]
    &&&&&&& &&&c&0 \\
    &&&&&&& &&&0&-c \\
    &&&&&&& &c&0&& \\
    &&&&&&& &0&-c&& \\
    \end{pNiceMatrix}
    }\text{\raisebox{-34pt}{,}}
  \end{align*}
  and we also have 
  \vspace{3mm}
\begin{gather*}
\hspace{3mm}\left(\Phi^*\omega_{\text{std}}\right)_{\Phi^{-1}(z)} = 
 {\addtolength{\arraycolsep}{-3pt}
    \begin{pNiceMatrix}
    \Block{2-7}{O^{2n+2}}&&&&&&&\Vdots&&&& \\
    &&&&&&&&&&\\
    \myhdottedline\\[-10pt]
    &&&&&&& &0&1&& \\
    &&&&&&& &-1&0&& \\
    &&&&&&& &&&0&1 \\
    &&&&&&& &&&-1&0 \\
    \end{pNiceMatrix}
    }\text{\raisebox{-34pt}{.}}
\end{gather*}

\item A basis for $T_{\Phi^{-1}(z)}\mathbb{R}^{2n+6}$ can be taken as $\{e_1,\ldots,e_{2n+6}\}$. Here, $\{e_1, \ldots, e_{2n+2}\}$ constitutes a basis for $d\Phi^{-1}(X_K)_z$.
Therefore, the representation matrices $d^2J, d^2H$, and $\Omega$ are given by 
\begin{gather*}
  d^2J=
  {\addtolength{\arraycolsep}{-3pt}
  \begin{pmatrix}
    -1&&&\\
    &-1&&\\
    &&1&\\
    &&&1
  \end{pmatrix}
  }
  \text{\raisebox{-19pt}{,}}\hspace{5mm}
  d^2H=
  {\addtolength{\arraycolsep}{-3pt}
  \begin{pNiceMatrix}
    &&c&0\\
    &&0&-c\\
    c&0&&\\
    0&-c&&
  \end{pNiceMatrix}
  }
  \text{\raisebox{-19pt}{,}}\hspace{5mm}
  \Omega=
    {\addtolength{\arraycolsep}{-3pt}
  \begin{pNiceMatrix}
    0&1&&\\
    -1&0&&\\
    &&0&1\\
    &&-1&0
  \end{pNiceMatrix}
  }\text{\raisebox{-19pt}{.}}
\end{gather*}

\item Since $c\neq0$, $d^2J$ and $d^2H$ are linearly independent.

\item If we put $a_1, a_2=1$, the set of solutions to the characteristic equation $\det\left((a_1d^2J+a_2d^2H)\Omega^{-1}-tI\right)=0$ is $\{\pm c\pm i\}$ (for all combinations of signs).

\item Therefore, the singularity $[p^{2n+3, 2n+4}]$ is of focus-focus type.
\end{enumerate}
We provide a remark regarding the calculation process described above. Even in cases other than $(k,l)=(2n+3,2n+4)$, by a suitable choice of a basis for $T_{\Phi^{-1}(z)}\mathbb{R}^{2n+6}$, $d^2J$ and $\Omega$ remain identical to their counterparts in Step\,$3$ of the preceding calculation, while $d^2H$ merely involves a different non-zero real value for $c$. Consequently, for all $k$ and $l$ that are neither $2n+1$ nor $2n+2$, the singularities $[p^{k,l}]$ are of focus-focus type.
\vspace{10mm}

\subsubsection{$[z]=[p^{k,l}]$ $\mathrm{(}$The case where $k$ or $l$ is $2n+1$ or $2n+2 \mathrm{)}$}\mbox{}\\

Each point $[z]$ is a singularity of rank\,$1$. As in the previous case, since the quantities appearing in the calculation are similar across all $[p^{k,l}]$, we describe the computation for the case $(k,l)=(2n+2,2n+4)$ in detail here.

As before, we can take a representative $z$ for $[z]$ which is real and zero on the $(2n+2)$-th and $(2n+4)$-th components.
\begin{enumerate}[label=Step\,\arabic*.]
  \item Since $dJ_{[z]}=0$, we may set $g=J$ and $\tilde{g}=\tilde{J}$.
  
  \item Solving the equation $dF_{\tilde{g}}(z,\bm{\lambda})=0$ for $\bm{\lambda}=(\lambda_1, \ldots, \lambda_{2n+2})$, we obtain 
  \begin{gather*}
  \lambda_j =
     \begin{cases}
       1 & (j=2n+1),\\
       -1 & (j=2n+2),\\
       0 & (otherwise).
     \end{cases}
   \end{gather*}

  \item $d^2F_{\tilde{g}}(z)$ is given as follows:
  \vspace{3mm}
  \begin{gather*}
    d^2F_{\tilde{g}}(z) =
    {\addtolength{\arraycolsep}{-3pt}
    \begin{pNiceMatrix}
      \Block{2-7}{O^{4n+2}}&&&&&&&&&&&&&&\\
      &&&&&& &&&&&&&&\\
      &&&&&& &1&&&&&&&\\
      &&&&&& &&1&&&&&&\\
      &&&&&& &&&\Block{2-6}{O^4}&&&&&\\
      &&&&&& &&&&&&&&
      \CodeAfter
      \tikz \draw [dotted]
      (3-|1) -- (3-|10) -- (7-|10)
      (1-|8) -- (5-|8) -- (5-|16);
    \end{pNiceMatrix}
    }
    \text{\raisebox{-32pt}{.}}\\
  \end{gather*}
  Next, we compute $d\varphi_{\Phi^{-1}(z)}$. The matrix $\Bigl(\pdv{v}{\eta}\Bigr)$ consisting of the non-zero column vectors of $dv$ at $z$ is non-singular. Therefore, by the implicit function theorem, there exists a map $\varphi=(\varphi_1,\ldots,\varphi_{2n+2})$ from an open neighborhood of $0\in \mathbb{R}^{2n+6}$ to an open neighborhood of $(x_1,\ldots,x_{2n+1}, x_{2n+3})\in \mathbb{R}^{2n+2}$ such that the map $\Phi=(\varphi_1, \xi_1, \ldots, \varphi_{2n+1}, \xi_{2n+1},\xi_{2n+2},\xi_{2n+3},\varphi_{2n+2}, \allowbreak \xi_{2n+4},\xi_{2n+5}, \xi_{2n+6})$ provides a chart for $v^{-1}(0)$ around the point $z$. Moreover, in this case, since $\Bigl(\pdv{v}{\xi}\Bigr)=0$, we see that $d\varphi_{\Phi^{-1}(z)}=0$. Therefore, we have 
  \begin{gather*}
   \hspace{3mm} d^2(\tilde{g}\circ\Phi)(z) =
    {\addtolength{\arraycolsep}{-3pt}
    \begin{pNiceMatrix}
      \Block{2-7}{O^{2n+1}}&&&&&&&&&&&\\
      &&&&&&&&&&&&&&\\
      &&&&&&&1&&&&&&&\\
      &&&&&&&&1&&&&&&\\
      &&&&&&&&&\Block{2-6}{O^3}&&&&&\\
      &&&&&&&&&&&&&&
      \CodeAfter
      \tikz \draw [dotted]
      (3-|1) -- (3-|10) -- (7-|10)
      (1-|8) -- (5-|8) -- (5-|16);
    \end{pNiceMatrix}
    }
    \text{\raisebox{-32pt}{,}}
  \end{gather*}
  and we also have 
  \begin{gather*}
\left(\Phi^*\omega_{\text{std}}\right)_{\Phi^{-1}(z)} = 
{\addtolength{\arraycolsep}{-3pt}
\begin{pNiceMatrix}
    \Block{2-7}{O^{2n+1}}&&&&&&&\Vdots&&&&&\\
    &&&&&&& &&&&&\\
    \myhdottedline\\[-10pt]
    &&&&&&& &0&1&&& \\
    &&&&&&& &-1&0&&& \\
    &&&&&&& &&&0&& \\
    &&&&&&& &&&&0&1 \\
    &&&&&&& &&&&-1&0 \\
    \end{pNiceMatrix}
    }\text{\raisebox{-40pt}{.}}\\[-0.5ex]
\end{gather*}

\item Since $dJ_{[z]}=0$ and $\text{rank}\,dF_{[z]}=1$, $L_{[z]}\subset T_{[z]}M_n$ is spanned by $X_H([z])$. We take a basis of $T_{\Phi^{-1}(z)}\mathbb{R}^{2n+6}$ as $\{e_1,\ldots,e_{2n+1},e_{2n+4},e_{2n+6},e_{2n+2},e_{2n+3},\allowbreak e_{2n+5}\}$, where $\{e_j\}_{1\leq j\leq2n+6}$ denotes the set of the standard basis vectors for $\mathbb{R}^{2n+6}$. Note that the first $2n+2$ vectors form a basis of $d\Phi^{-1}(X_K)_z$  and we can take $e_{2n+6}$ instead of $d\Phi^{-1}(X_{\tilde{H}})(z)$ because $e_{2n+6}$ is linearly dependent on $d\Phi^{-1}(X_{\tilde{H}})(z)$. With respect to this basis of $T_{\Phi^{-1}(z)}\mathbb{R}^{2n+6}$, the $4$-by-$4$ submatrices, $d^2G$ and $\Omega$, from the $(2n+3,2n+3)$ to $(2n+6,2n+6)$ entries of the matrices representing $d^2(\tilde{g}\circ \Phi)_{\Phi^{-1}(z)}$ and $(\Phi^*\omega_{\text{std}})_{\Phi^{-1}(z)}$, respectively, are given as follows:\vspace{2mm}
\begin{gather*}
  d^2G =
  {\addtolength{\arraycolsep}{-3pt}
  \begin{pmatrix}
    0&&&\\
    &1&&\\
    &&1&\\
    &&&0
  \end{pmatrix}
  }
  \text{\raisebox{-19pt}{,}}\hspace{5mm}
  \Omega =
  {\addtolength{\arraycolsep}{-3pt}
  \begin{pmatrix}
    &&0&-1\\
    &&1&0\\
    0&-1&&\\
    1&0&&
  \end{pmatrix}
  }\text{\raisebox{-20pt}{.}}
\end{gather*}

\item Therefore, we have
\begin{gather*}
\widehat{d^2G}=
{\addtolength{\arraycolsep}{-3pt}
  \begin{pmatrix}
    1&0\\
    0&1
  \end{pmatrix}
  }
  \text{\raisebox{-7pt}{,}}\hspace{5mm}
  \widehat{\Omega}=
  {\addtolength{\arraycolsep}{-3pt}
  \begin{pmatrix}
    0&1\\
    -1&0
  \end{pmatrix}
  }\text{\raisebox{-8pt}{.}}
\end{gather*}
Then, the solution set of the characteristic equation $\det\left(\widehat{d^2G}\,\widehat{\Omega}^{-1}-tI\right)=0$ is $\{\pm i\}$.

\item Consequently, the point $[z]$ is a singularity of Elliptic-Regular type.
\end{enumerate}

For other choices of $k$ and $l$, the matrices $\widehat{d^2G}$ and $\widehat{\Omega}$ obtained in Step $5$ are identical to those in the computation example above. Therefore, the points $[p^{k,l}]$ for which either $k$ or $l$ is $2n+1$ or $2n+2$ are singularities of Elliptic-Regular type.
\vspace{9mm}

This completes the determination of all types of the singularities $[z]=[p^{k,l}]$. Next, we determine the types of the singularities $[z]=[p^k]$. In this case, we recall that a point $[z]$ is a singularity if and only if $k=2n+1$ or $2n+2$, and in both case, the singularity $[z]$ can be of rank\,$0$ or $1$ (see Table\,\ref{table:sing}). In this part, we describe the details of the calculation only for the case $k=2n+1$ (the case $k=2n+2$ is completely analogous).

\subsubsection{$[z]=[p^{2n+1}]$ $\mathrm{(rank\,1)}$}\mbox{}\\

A representative of the point $[z]$ can be taken as $z=(x_1,\ldots,x_{2n},0,x_{2n+2},x_{2n+3},\allowbreak z_{2n+4})\in v^{-1}(0)$ with $x_j\in \mathbb{R}$ and $z_{2n+4}\in \mathbb{C}$. 

\begin{enumerate}[label=Step\,\arabic*. ]
  \item Since $dJ_{[z]}=0$, we may set $g=J$ and $\tilde{g}=\tilde{J}$. 
  \item Solving the equation $dF_{\tilde{g}}(z,\lambda)=0$ for $\lambda=(\lambda_1, \ldots, \lambda_{2n+2})$, we obtain 
 \begin{gather*}
  \lambda_j =
     \begin{cases}
       -1 & (j=2n+2),\\
       0 & (otherwise).
     \end{cases}
   \end{gather*}

  \item $d^2F_{\tilde{g}}(z)$ is given as follows:
  \begin{gather*}
    d^2F_{\tilde{g}}(z) =
    {\addtolength{\arraycolsep}{-3pt}
    \begin{pNiceMatrix}[margin]
      \Block{2-6}{O^{4n}}&&&&&&&&&&&&&\\
      &&&&&&&&&&&&&\\
      &&&&&&-1&&&&&&&\\
      &&&&&&&-1&&&&&&\\
      &&&&&&&&\Block{2-6}{O^{6}}&&&&&\\
      &&&&&&&&&&&&&
      \CodeAfter
      \tikz \draw [dotted]
      (1-|7) -- (5-|7) -- (5-|15)
      (3-|1) -- (3-|9) -- (7-|9);
    \end{pNiceMatrix}
    }
    \text{\raisebox{-32pt}{.}}
  \end{gather*}
Next, we compute $d\varphi_{\Phi^{-1}(z)}$. Since the matrix $\Bigl(\pdv{v}{\eta}\Bigr)$ obtained by collecting all non-zero column vectors of $dv_z$ excluding the last two columns is invertible, the implicit function theorem implies that there exists a map $\varphi=(\varphi_1,\ldots,\varphi_{2n+2})$ from an open neighborhood of $(0,\ldots,0,x_{2n+4},y_{2n+4})\in \mathbb{R}^{2n+6}$ to an open neighborhood of $(x_1,\ldots,x_{2n},x_{2n+2},x_{2n+3})\in \mathbb{R}^{2n+2}$ such that the map $\Phi=(\varphi_1,\xi_1,\ldots,\varphi_{2n},\xi_{2n},\xi_{2n+1},\xi_{2n+2},\varphi_{2n+1},\xi_{2n+3},\varphi_{2n+2},\allowbreak \xi_{2n+4},\xi_{2n+5},\xi_{2n+6})$ provides a chart of $v^{-1}(0)$ around the point $z$. Furthermore, letting $\Bigl(\pdv{v}{\xi}\Bigr)$ be the matrix obtained by removing $\Bigl(\pdv{v}{\eta}\Bigr)$ from the matrix $dv_z$, since $d\varphi_{\Phi^{-1}(z)} = -\left(\pdv{v}{y}\right)^{-1}\Bigl(\pdv{v}{\xi}\Bigr)$ (as mentioned before Lemma \ref{lem:laghessian}), we obtain  
  \begin{gather*}
  d\varphi_{\Phi^{-1}(z)}=
  {\addtolength{\arraycolsep}{-2pt}
    \begin{pNiceMatrix}
    \Block{5-10}{O^{2n+2,2n+4}}&&&&&&&&&&\Vdots&-\dfrac{x_{2n+4}}{x_1}&-\dfrac{y_{2n+4}}{x_1}\\
    &&&&&&&&&& &\vdots&\vdots\\
    &&&&&&&&&& &-\dfrac{x_{2n+4}}{x_{2n}}&-\dfrac{y_{2n+4}}{x_{2n}}\\
    &&&&&&&&&& &0&0\\
    &&&&&&&&&& &\dfrac{x_{2n+4}}{x_{2n+3}}&\dfrac{y_{2n+4}}{x_{2n+3}}\\
  \end{pNiceMatrix}\text{\raisebox{-41pt}{.}}
  }\\
\end{gather*}
  Therefore, we also have
  \begin{gather*}
    d^2(\tilde{g}\circ\Phi)(z) =
    {\addtolength{\arraycolsep}{-3pt}
    \begin{pNiceMatrix}
      \Block{2-6}{O^{2n}}&&&&&&&&&&&&&\\
      &&&&&&&&&&&&&\\
      &&&&&&-1&&&&&&&\\
      &&&&&&&-1&&&&&&\\
      &&&&&&&&\Block{2-6}{O^4}&&&&&\\
      &&&&&&&&&&&&&
      \CodeAfter
      \tikz \draw [dotted]
      (1-|7) -- (5-|7) -- (5-|15)
      (3-|1) -- (3-|9) -- (7-|9);
    \end{pNiceMatrix}
    }
    \text{\raisebox{-32pt}{,}}\\
    \end{gather*}
  and
  \begin{gather*}
\left(\Phi^*\omega_{\text{std}}\right)_{\Phi^{-1}(z)} = 
{\addtolength{\arraycolsep}{-3pt}
\begin{pNiceMatrix}
    \Block{2-5}{A_{\omega}}&&&&&\Vdots&\Block{2-2}{B_{\omega}}&\\
    &&&&& &&\\
    \myhdottedline\\[-10pt]
    \Block{2-5}{\raisebox{15pt}{$-{}^t\!B_{\omega}$}}&&&&& &0&1\\
    &&&&& &-1&0
    \end{pNiceMatrix}\text{\raisebox{-23pt}{.}}
    }\\
\end{gather*}
Here, $A_{\omega}$ is a $(2n+4)$-by-$(2n+4)$ submatrix of $\bigl(\Phi^*\omega_{\text{std}}\bigr)_{\Phi^{-1}(z)}$ whose entries from the $(2n+1,2n+1)$-th to the $(2n+2,2n+2)$-th positions form the block $\bigl(\begin{smallmatrix} 0&1\\-1&0 \end{smallmatrix}\bigr)$ and all other entries are zero, and $B_{\omega}$ is given by \vspace{3mm}
\begin{gather*}
B_{\omega}=
{\addtolength{\arraycolsep}{-2pt}
  \begin{pNiceMatrix}
    \dfrac{x_{2n+4}}{x_1}&\dfrac{y_{2n+4}}{x_1}\\
    \vdots&\vdots\\
    \dfrac{x_{2n+4}}{x_{2n}}&\dfrac{y_{2n+4}}{x_{2n}}\\
    0&0\\
    0&0\\
    0&0\\
    -\dfrac{x_{2n+4}}{x_{2n+3}}&-\dfrac{y_{2n+4}}{x_{2n+3}}\\
  \end{pNiceMatrix}\text{\raisebox{-52pt}{.}}
  }\\
\end{gather*}

\item Since $dJ_{[z]}=0$ and $\text{rank}\,dF_{[z]}=1$, $L_{[z]}\subset T_{[z]}M_n$ is spanned by $X_H([z])$. 
In the following, only in Step $4$, we divide the argument into cases depending on whether the point $z$ satisfies $y_{2n+4}=0$ or not, which are presented in Step $4.1$ and Step $4.2$, respectively. From Step $5$, we can deal with both cases together, so we do not consider separate cases.
    \item[\theenumi1.] Now, the matrix whose columns form a basis of $T_{\Phi^{-1}(z)}\mathbb{R}^{2n+6}$ is given by
    \begin{gather*}
    {\addtolength{\arraycolsep}{-3pt}
  \begin{pNiceMatrix}
    x_1&&&&&&\Vdots&R_1&&&\\
    &x_2&&&&& &R_2&&&\\
    &&\ddots&&&& &\vdots&&&\\
    &&&x_{2n}&&& &R_{2n}&&&\\
    &&&&&& &&1&&\\
    &&&&&& &&&1&\\
    \myhdottedline\\[-10pt]
    &&&&x_{2n+2}&& &&&&\\
    &&&&&-x_{2n+3}& &R_{2n+3}&&&\\
    &&&&&& &&&&1\\
    x_{2n+4}&x_{2n+4}&\cdots&x_{2n+4}&&x_{2n+4}& &R_{2n+4}&&&
  \end{pNiceMatrix}\text{\raisebox{-62pt}{.}}
  }
  \end{gather*}
  Then, the set of column vectors to the left of the vertical dotted line forms a basis of $d\Phi^{-1}(X_K)_z$, while the column vector to the right of the line is $d\Phi^{-1}\bigl(X_{\tilde{H}}(z)\bigr)$.  With respect to this basis of $T_{\Phi^{-1}(z)}\mathbb{R}^{2n+6}$, the $4$-by-$4$ submatrices, $d^2G$ and $\Omega$, from the $(2n+3,2n+3)$ to $(2n+6,2n+6)$ entries of the matrices representing $d^2(\tilde{g}\circ \Phi)_{\Phi^{-1}(z)}$ and $(\Phi^*\omega_{\text{std}})_{\Phi^{-1}(z)}$, respectively, are given as follows:
\begin{gather*}
  d^2G =
  {\addtolength{\arraycolsep}{-3pt}
  \begin{pmatrix}
    0&&&\\
    &-1&&\\
    &&-1&\\
    &&&0
  \end{pmatrix}
  }
  \text{\raisebox{-18pt}{,}}\hspace{5mm}
  \Omega =
  {\addtolength{\arraycolsep}{-3pt}
  \begin{pmatrix}
    &&&c\\
    &0&1&\\
    &-1&0&\\
    -c&&&
  \end{pmatrix}
  }\text{\raisebox{-18pt}{,}}
\end{gather*}
where $c\neq0$.
  
    \item[\theenumi2.] Similarly to the previous case, in the following matrix, 
\begin{gather*}
    {\addtolength{\arraycolsep}{-4pt}
    \begin{pNiceMatrix}
    x_1&&&&&&\Vdots&R_1&&&\\
    &x_2&&&&& &R_2&&&\\
    &&\ddots&&&& &\vdots&&&\\
    &&&x_{2n}&&& &R_{2n}&&&\\
    &&&&&& &&1&&\\
    &&&&&& &&&1&\\
    \myhdottedline\\[-10pt]
    &&&&x_{2n+2}&& &&&&\\
    &&&&&-x_{2n+3}& &R_{2n+3}&&&\\
    -y_{2n+4}&-y_{2n+4}&\ldots&-y_{2n+4}&&-y_{2n+4}& &&&&-y_{2n+4}\\
    x_{2n+4}&x_{2n+4}&\ldots&x_{2n+4}&&x_{2n+4}& &R_{2n+4}&&&x_{2n+4}
  \end{pNiceMatrix}\text{\raisebox{-60pt}{,}}
  }
  \end{gather*}
  the set of column vectors to the left of the vertical dotted line forms a basis of $d\Phi^{-1}(X_K)_z$, and the column vector immediately to the right of the line is $d\Phi^{-1}\bigl(X_{\tilde{H}}(z)\bigr)$. With respect to this basis of $T_{\Phi^{-1}(z)}\mathbb{R}^{2n+6}$, the $4$-by-$4$ submatrices, $d^2G$ and $\Omega$, from the $(2n+3,2n+3)$ to $(2n+6,2n+6)$ entries of the matrices representing $d^2(\tilde{g}\circ \Phi)_{\Phi^{-1}(z)}$ and $(\Phi^*\omega_{\text{std}})_{\Phi^{-1}(z)}$, respectively, are given as follows:
\begin{gather*}
  d^2G =
  {\addtolength{\arraycolsep}{-3pt}
  \begin{pmatrix}
    0&&&\\
    &-1&&\\
    &&-1&\\
    &&&0
  \end{pmatrix}
  }
  \text{\raisebox{-18pt}{,}}\hspace{5mm}
  \Omega =
  {\addtolength{\arraycolsep}{-3pt}
  \begin{pmatrix}
    &&&\text{Im}f\\
    &0&1&\\
    &-1&0&\\
      -\text{Im}f&&&
  \end{pmatrix}
  }\text{\raisebox{-20pt}{.}}
\end{gather*}
\item Since $L_{[z]}$ is generated by $e_1$, we have $L_{[z]}^{\omega}=\langle e_1,e_2,e_3 \rangle$. Therefore, $\widehat{d^2G}$ and $\widehat{\Omega}$ are as follows:
\begin{gather*}
\widehat{d^2G}=
{\addtolength{\arraycolsep}{-3pt}
  \begin{pmatrix}
    -1&0\\
    0&-1
  \end{pmatrix}
  }
  \text{\raisebox{-6pt}{,}}\hspace{5mm}
  \widehat{\Omega}=
  {\addtolength{\arraycolsep}{-3pt}
  \begin{pmatrix}
    0&1\\
    -1&0
  \end{pmatrix}
  }\text{\raisebox{-6pt}{,}}
\end{gather*}
and the solution set of the characteristic equation $\det \bigl(\widehat{d^2G}\,\widehat{\Omega}^{-1}-tI\bigr)=0$ is $\{\pm i\}$.

\item Consequently, the point $[z]$ is a singularity of Elliptic-Regular type.
\end{enumerate}

\subsubsection{$[z]=[p^{2n+1}]$ $\mathrm{(rank\,0)}$}\mbox{}\\

A representative of the point $[z]$ can be taken as $z=(x_1,\ldots,x_{2n},0,x_{2n+2},x_{2n+3},\allowbreak x_{2n+4})\in v^{-1}(0)$ with $x_j\in \mathbb{R}$.
Before describing the calculation process, we make a few remarks. The matrices related to $\tilde{J}$, $J$, $\varphi$, $\Phi$, and $\bigl(\Phi^*\omega_{\text{std}}\bigr)_{\Phi^{-1}(z)}$ appearing in Step\,$1$ and Step\,$2$ below can be chosen to be identical to their counterparts in the preceding calculation process. Therefore, in Steps $1$ and $2$ below, we only compute the quantities related to $\tilde{H}$ and $H$.

\begin{enumerate}[label=Step\,\arabic*. ]
\item Solving the equation $dF_{\tilde{H}}(z, \bm{\lambda'}) = 0$ for $\bm{\lambda'}=(\lambda'_1,\ldots,\lambda'_{2n+2})$, we obtain
\begin{gather*}
  \lambda'_j =
  \begin{cases}
    -\dfrac{H}{|z_j|^2} &(j=1,\ldots,2n),\\[2ex]
    \sum\limits_{k=1}^n(k-1)\dfrac{H}{|z_k|^2} &(j=2n+1),\\[2ex]
    \left(-\dfrac{1}{|z_{2n+4}|^2}+\sum\limits_{k=1}^n\dfrac{1}{|z_k|^2}\right)H &(j=2n+2).
  \end{cases}
\end{gather*}

\item The $d^2F_{\tilde{H}}(z)$ is given by
\begin{gather*}
d^2F_{\tilde{H}}(z) =
\begin{pNiceMatrix}
    \Block{2-4}{A_{2n+1}}&&&&\Vdots&\Block{2-2}{O^{4n,4}}&&\Vdots&\Block{2-4}{B_{2n+1}}&&&\\
     &&&&&&&&&&& \\
     \myhdottedline\\[-10pt]
    \Block{2-4}{\raisebox{10pt}{$O^{4,4n}$}}&&&& &\Block{2-2}{\raisebox{10pt}{$C_{2n+1}$}}&& &\Block{2-4}{\raisebox{10pt}{$O^4$}}&&&\\
     &&&&&&&&&&& \\
     \myhdottedline\\[-10pt]
    \Block{2-4}{\raisebox{12pt}{${}^t\!B_{2n+1}$}}&&&& &\Block{2-2}{\raisebox{10pt}{$O^4$}}&& &\Block{2-4}{\raisebox{10pt}{$D_{2n+1}$}} &&& \\
     &&&&&&&&&&&& \\
    \end{pNiceMatrix}\text{\raisebox{-33pt}{,}}
\end{gather*}
where
\begin{align*}
  A_{2n+1}&=
  {\addtolength{\arraycolsep}{-3pt}
  \begin{pmatrix}
    -\dfrac{H}{|z_1|^2}&0&R_{1,2}&0&\cdots&R_{1,2n}&0\\
    0&-\dfrac{H}{|z_1|^2}&0&-R_{1,2}&\cdots&0&-R_{1,2n}\\
    R_{1,2}&0&-\dfrac{H}{|z_2|^2}&0&\cdots&R_{2,2n}&0\\
    0&-R_{1,2}&0&-\dfrac{H}{|z_2|^2}&\cdots&0&-R_{2,2n}\\
    \vdots&\vdots&\vdots&\vdots&\ddots&\vdots&\vdots\\
    R_{1,2n}&0&R_{2,2n}&0&\cdots&-\dfrac{H}{|z_{2n}|^2}&0\\
    0&-R_{1,2n}&0&-R_{2,2n}&\cdots&0&-\dfrac{H}{|z_{2n}|^2}
  \end{pmatrix}\text{\raisebox{-72pt}{,}}
  }\\[3ex]
   B_{2n+1}&=
  {\addtolength{\arraycolsep}{-3pt}
  \begin{pmatrix}
    R_{1,2n+3}&0&R_{1,2n+4}&0\\
    0&R_{1,2n+3}&0&R_{1,2n+4}\\
    R_{2,2n+3}&0&R_{2,2n+4}&0\\
    0&R_{2,2n+3}&0&R_{2,2n+4}\\
    \vdots&\vdots&\vdots&\vdots\\
    R_{2n,2n+3}&0&R_{2n,2n+4}&0\\
    0&R_{2n,2n+3}&0&R_{2n,2n+4}\\
  \end{pmatrix}\text{\raisebox{-40pt}{,}}
  }\hspace{4mm}
  C_{2n+1}=
  {\addtolength{\arraycolsep}{-3pt}
  \begin{pmatrix}
    \kappa_1&&&\\
    &\kappa_1&&\\
    &&0&\\
    &&&0
  \end{pmatrix}\text{\raisebox{-18pt}{,}}
  }\\[3ex]
  D_{2n+1}&=
   {\addtolength{\arraycolsep}{-3pt}
  \begin{pmatrix}
    -\dfrac{H}{|z_{2n+3}|^2}&&R_{2n+3,2n+4}&\\
    &-\dfrac{H}{|z_{2n+3}|^2}&&-R_{2n+3,2n+4}\\
    R_{2n+3,2n+4}&&-\dfrac{H}{|z_{2n+4}|^2}&\\
    &-R_{2n+3,2n+4}&&-\dfrac{H}{|z_{2n+4}|^2}
  \end{pmatrix}\text{\raisebox{-40pt}{,}}
  }\\[3ex]
  \kappa_1&=\Biggl(-\dfrac{1}{|z_{2n+4}|^2}+\sum_{k=1}^n\left(\dfrac{k}{|z_k|^2}+\dfrac{1-k}{|z_{n+k}|^2}\right)\Biggr)H.
\end{align*}
Therefore, we have 
\begin{gather*}
  d^2(\tilde{H}\circ\Phi)=
  \begin{pNiceMatrix}
    \Block{2-4}{\tilde{A}_{2n+1}}&&&&\Vdots&\Block{2-2}{O^{2n,3}}&&\Vdots&\Block{2-4}{\tilde{B}_{2n+1}}&&&\\
     &&&&&&&&&&& \\
     \myhdottedline\\[-10pt]
    \Block{2-4}{\raisebox{10pt}{$O^{3,2n}$}}&&&& &\Block{2-2}{\raisebox{10pt}{$\tilde{C}_{2n+1}$}}&& &\Block{2-4}{\raisebox{10pt}{$O^3$}}&&&\\
     &&&&&&&&&&& \\
     \myhdottedline\\[-10pt]
    \Block{2-4}{\raisebox{12pt}{${}^t\!\tilde{B}_{2n+1}$}}&&&& &\Block{2-2}{\raisebox{10pt}{$O^3$}}&& &\Block{2-4}{\raisebox{10pt}{$\tilde{D}_{2n+1}$}} &&& \\
     &&&&&&&&&&&& \\
    \end{pNiceMatrix}\text{\raisebox{-32pt}{,}}
\end{gather*}
where
\begin{align*}
\tilde{A}_{2n+1}&=
{\addtolength{\arraycolsep}{-3pt}
\begin{pmatrix}
  -\dfrac{H}{|z_1|^2}&-R_{1,2}&\cdots&-R_{1,2n}\\
  -R_{1,2}&-\dfrac{H}{|z_2|^2}&\cdots&-R_{2,2n}\\
  \vdots&\vdots&\ddots&\vdots\\
  -R_{1,2n}&-R_{2,2n}&\cdots&-\dfrac{H}{|z_{2n}|^2}
\end{pmatrix}\text{\raisebox{-36pt}{,}}
}\\[3ex]
\tilde{B}_{2n+1}&=
{\addtolength{\arraycolsep}{-3pt}
\begin{pmatrix}
  R_{1,2n+3}&0&R_{1,2n+4}\\
  R_{2,2n+3}&0&R_{2,2n+4}\\
  \vdots&\vdots&\vdots\\
  R_{2n,2n+3}&0&R_{2n,2n+4}\\
\end{pmatrix}\text{\raisebox{-21pt}{,}}
}\hspace{4mm}
\tilde{C}_{2n+1}=
{\addtolength{\arraycolsep}{-3pt}
\begin{pmatrix}
\kappa_1&&\\
&\kappa_1&\\
&&0
\end{pmatrix}\text{\raisebox{-14pt}{,}}
}\\[3ex]
\tilde{D}_{2n+1}&=
{\addtolength{\arraycolsep}{-3pt}
\begin{pmatrix}
-\dfrac{H}{|z_{2n+3}|^2}&&-R_{2n+3,2n+4}\\
&\kappa_2&\\
-R_{2n+3,2n+4}&&-\dfrac{H}{|z_{2n+4}|^2}
\end{pmatrix}\text{\raisebox{-21pt}{,}}
}\\[3ex]
\kappa_2&=-2|z_{2n+4}|^2\left(\dfrac{1}{|z_{2n+3}|^4}+\dfrac{1}{|z_{2n+4}|^4}+\sum\limits_{j=1}^{2n}\dfrac{1}{|z_j|^4}\right)H.
\end{align*}
\item Now, we can choose the following matrix whose columns form a basis of $T_{\Phi^{-1}(z)}\mathbb{R}^{2n+6}$:\vspace{2mm}
\begin{gather*}
    {\addtolength{\arraycolsep}{-3pt}
  \begin{pNiceMatrix}
    x_1&&&&&&\Vdots&&&&\\
    &x_2&&&&& &&&&\\
    &&\ddots&&&& &&&&\\
    &&&x_{2n}&&& &&&&\\
    \myhdottedline\\[-10pt]
    &&&&&& &1&&&\\
    &&&&&& &&1&&\\
    \myhdottedline\\[-10pt]
    &&&&x_{2n+2}&& &&&&\\
    &&&&&-x_{2n+3}& &&&&\\
    &&&&&& &&&1&\\
    x_{2n+4}&x_{2n+4}&\ldots&x_{2n+4}&&x_{2n+4}& &&&&1
  \end{pNiceMatrix}\text{\raisebox{-64pt}{.}}
  }\\[-0.5ex]
  \end{gather*}
  The column vectors to the left of the vertical dotted line in the matrix above form a basis of $d\Phi^{-1}(X_K)_z$. Therefore, a straightforward computation shows that $d^2J, d^2H$, and $\Omega$ are given by the following matrices:
  \begin{gather*}
  d^2J=
  {\addtolength{\arraycolsep}{-3pt}
  \begin{pmatrix}
    -1&&&\\
    &-1&&\\
    &&0&\\
    &&&0
  \end{pmatrix}
  }
  \text{\raisebox{-19pt}{,}}\hspace{5mm}
  d^2H=
  {\addtolength{\arraycolsep}{-3pt}
  \begin{pNiceMatrix}
    \kappa_1&&&\\
    &\kappa_1&&\\
    &&\kappa_2&\\
    &&&-\dfrac{H}{|z_{2n+4}|^2}
  \end{pNiceMatrix}
  }
  \text{\raisebox{-21pt}{,}}\hspace{5mm}
  \Omega=
    {\addtolength{\arraycolsep}{-3pt}
  \begin{pNiceMatrix}
    0&1&&\\
    -1&0&&\\
    &&0&1\\
    &&-1&0
  \end{pNiceMatrix}
  }\text{\raisebox{-19pt}{.}}
\end{gather*}

\item Since $\text{Im}\,f=0$, $-\dfrac{H}{|z_{2n+4}|^2}$ is non-zero. Therefore, $d^2J$ and $d^2H$ are linearly independent.

\item Depending on whether $\kappa_1=0$, we can choose the constants $a_1$ and $a_2$ appropriately so that $a_1d^2J+a_2d^2H$ is invertible. Now, we take $(a_1,a_2)=(1,1)$ if $\kappa_1=0$, and $(a_1,a_2)=(0,1)$ otherwise. In each case, $P(t)$ becomes 
\begin{align*}
  P(t)&=(t^2+1)\Biggl(t^2+2\Bigl(\dfrac{1}{|z_{2n+3}|^4}+\dfrac{1}{|z_{2n+4}|^4}+\sum\limits_{j=1}^{2n} \dfrac{1}{|z_j|^4}\Bigr)H^2\Biggr)\hspace{3mm}&(\kappa_1=0),\\[2ex]
  P(t)&=(t^2+\kappa_1^2)\Biggl(t^2+2\Bigl(\dfrac{1}{|z_{2n+3}|^4}+\dfrac{1}{|z_{2n+4}|^4}+\sum\limits_{j=1}^{2n} \dfrac{1}{|z_j|^4}\Bigr)H^2\Biggr)\hspace{3mm}&(\kappa_1\neq0).
\end{align*}

\item Therefore, in both cases, the point $[z]$ is a singularity of Elliptic-Elliptic type.
\end{enumerate}

\vspace{5mm}
\subsubsection{$[z]=[p]$} \mbox{}\\

Finally, we present the calculation process for determining the types of the singularities when $[z]=[p]$. For the necessary and sufficient condition for a point $[p]$ to be a singularity, see Proposition \ref{prop:singtype}.

We can take a representative $z$ of the point $[z]$ to be $z=(x_1,\ldots,x_{2n+2},z_{2n+3},z_{2n+4})$. Namely, we may assume that all but the $(2n+3)$-th and $(2n+4)$-th components of $z$ are real numbers. 

\begin{enumerate}[label=Step\,\arabic*. ]
\item We can choose real numbers $b_1$ and $b_2$ satisfying the equation $b_1dJ_{[z]}+b_2dH_{[z]}=0$ as
\begin{gather*}
  b_1= \Biggl(\dfrac{1}{|z_{2n+4}|^2}-\sum_{k=1}^n\left(\dfrac{k}{|z_k|^2}+\dfrac{1-k}{|z_{n+k}|^2}\right)\Biggr)H, \hspace{5mm}b_2=-1.
\end{gather*}

\item Solving the equation $dF_{\tilde{g}}(z,\lambda)=0$ for $\lambda=(\lambda_1, \ldots, \lambda_{2n+2})$, we obtain 
\begin{gather*}
  \lambda_j =
     \begin{cases}
       \dfrac{H}{|z_j|^2} & (j=1,\ldots,2n),\\[2ex]
       \sum\limits_{k=1}^n \dfrac{1-k}{|z_k|^2}H& (j=2n+1),\\[2ex]
       \sum\limits_{k=1}^n\left(\dfrac{1}{|z_k|^2}-\dfrac{1}{|z_{n+k}|^2}\right)(k-1)H & (j=2n+2).
     \end{cases}
   \end{gather*}

  \item $d^2F_{\tilde{g}}(z)$ is given as follows:
  \begin{gather*}
  d^2F_{\tilde{g}}(z)=
\begin{pNiceMatrix}
    \Block{2-4}{\tilde{A}}&&&&\Vdots&\Block{2-2}{O^{4n,4}}&&\Vdots&\Block{2-2}{\tilde{B}}&\\
     &&&&&&&&& \\
     \myhdottedline\\[-10pt]
    \Block{2-4}{\raisebox{10pt}{$O^{4,4n}$}}&&&& &\Block{2-2}{\raisebox{10pt}{$O^4$}}&& &\Block{2-2}{\raisebox{10pt}{$O^4$}}&\\
     &&&&&&&&& \\
     \myhdottedline\\[-10pt]
    \Block{2-4}{\raisebox{12pt}{${}^t\!\tilde{B}$}}&&&& &\Block{2-2}{\raisebox{10pt}{$O^4$}}&& &\Block{2-2}{\raisebox{10pt}{$\tilde{C}$}} & \\
     &&&&&&&&&& \\
    \end{pNiceMatrix}\text{\raisebox{-32pt}{,}}
  \end{gather*}
where
\begin{align*}
  \tilde{A}&=
  {\addtolength{\arraycolsep}{-3pt}
  \begin{pNiceMatrix}
  \dfrac{H}{|z_1|^2}&0&-R_{1,2}&I_{1,2}&\cdots&-R_{1,2n}&I_{1,2n}\\
  0&\dfrac{H}{|z_1|^2}&I_{1,2}&R_{1,2}& \cdots &I_{1,2n}&R_{1,2n}\\
  -R_{1,2}&I_{1,2}&\dfrac{H}{|z_2|^2}&0& \cdots &-R_{2,2n}&I_{2,2n}\\
  I_{1,2}&R_{1,2}&0&\dfrac{H}{|z_2|^2}& \cdots &I_{2,2n}&R_{2,2n}\\
  \vdots&\vdots&\vdots&\vdots&\ddots&\vdots&\vdots\\
  -R_{1,2n}&I_{1,2n}&-R_{2,2n}&I_{2,2n}& \cdots &\dfrac{H}{|z_{2n}|^2}&0\\
  I_{1,2n}&R_{1,2n}&I_{2,2n}&R_{2,2n}& \cdots &0&\dfrac{H}{|z_{2n}|^2}
  \end{pNiceMatrix}\text{\raisebox{-71pt}{,}}
  }\\[3ex]
  \tilde{B}&=
  {\addtolength{\arraycolsep}{-3pt}
  \begin{pNiceMatrix}
    -R_{1,2n+3}&-I_{1,2n+3}&-R_{1,2n+4}&-I_{1,2n+4}\\
    I_{1,2n+3}&-R_{1,2n+3}&I_{1,2n+4}&-R_{1,2n+4}\\
    -R_{2,2n+3}&-I_{2,2n+3}&-R_{2,2n+4}&-I_{2,2n+4}\\
    I_{2,2n+3}&-R_{2,2n+3}&I_{2,2n+4}&-R_{2,2n+4}\\
    \vdots&\vdots&\vdots&\vdots\\
    -R_{2n,2n+3}&-I_{2n,2n+3}&-R_{2n,2n+4}&-I_{2n,2n+4}\\
    I_{2n,2n+3}&-R_{2n,2n+3}&I_{2n,2n+4}&-R_{2n,2n+4}\\
  \end{pNiceMatrix}\text{\raisebox{-39pt}{,}}
  }\\[3ex]
  \tilde{C}&=
  {\addtolength{\arraycolsep}{-3pt}
  \begin{pNiceMatrix}
  \dfrac{H}{|z_{2n+3}|^2}&&-R_{2n+3,2n+4}&\\
  &\dfrac{H}{|z_{2n+3}|^2}&&R_{2n+3,2n+4}\\
  -R_{2n+3,2n+4}&&\dfrac{H}{|z_{2n+4}|^2}&\\
  &R_{2n+3,2n+4}&&\dfrac{H}{|z_{2n+4}|^2}
  \end{pNiceMatrix}\text{\raisebox{-40pt}{.}}
  }
\end{align*}
 Next, we compute $d\varphi_{\Phi^{-1}(z)}$. Since the matrix $\Bigl(\pdv{v}{\eta}\Bigr)$ obtained by collecting all non-zero column vectors of $dv_z$ excluding the last four columns is invertible, the implicit function theorem implies that there exists a map $\varphi=(\varphi_1,\ldots,\varphi_{2n+2})$ from an open neighborhood of $(0,\ldots,0, x_{2n+3},y_{2n+3},x_{2n+4},\allowbreak y_{2n+4})\in \mathbb{R}^{2n+6}$ to an open neighborhood of $(x_1,\ldots,x_{2n+2})\in \mathbb{R}^{2n+2}$ such that the map $\Phi=(\varphi_1,\xi_1,\ldots,\varphi_{2n+2},\xi_{2n+2},\xi_{2n+3},\xi_{2n+4},\xi_{2n+5},\xi_{2n+6})$ provides a chart of $v^{-1}(0)$ around the point $z$. Moreover, letting $\Bigl(\pdv{v}{\xi}\Bigr)$ be the matrix obtained by removing $\Bigl(\pdv{v}{\eta}\Bigr)$ from the matrix $dv_z$, since $d\varphi_{\Phi^{-1}(z)} = -\left(\pdv{v}{y}\right)^{-1}\Bigl(\pdv{v}{\xi}\Bigr)$ (as mentioned before Lemma \ref{lem:laghessian}), we obtain $d\varphi_{\Phi^{-1}(z)}=\bigl(O^{2n+2}\,\text{\raisebox{-2pt}{\vdots}} \,B_{\varphi}\bigr)$, where $B_\varphi$ is the following $(2n+2)$-by-$4$ matrix:
 \begin{gather*}
 B_{\varphi}=
  {\addtolength{\arraycolsep}{-3pt}
   \begin{pNiceMatrix}
     0&0&-\dfrac{x_{2n+4}}{x_1}&-\dfrac{y_{2n+4}}{x_1}\\
     \dfrac{x_{2n+3}}{x_2}&\dfrac{y_{2n+3}}{x_2}&-\dfrac{2x_{2n+4}}{x_2}&-\dfrac{2y_{2n+4}}{x_2}\\
     \vdots&\vdots&\vdots&\vdots\\
     \dfrac{(j-1)x_{2n+3}}{x_j}&\dfrac{(j-1)y_{2n+3}}{x_j}&-\dfrac{jx_{2n+4}}{x_j}&-\dfrac{jy_{2n+4}}{x_j}\\
     \vdots&\vdots&\vdots&\vdots\\
     \dfrac{(n-1)x_{2n+3}}{x_n}&\dfrac{(n-1)y_{2n+3}}{x_n}&-\dfrac{nx_{2n+4}}{x_n}&-\dfrac{ny_{2n+4}}{x_n}\\[2ex]
     -\dfrac{x_{2n+3}}{x_{n+1}}&-\dfrac{y_{2n+3}}{x_{n+1}}&0&0\\
     -\dfrac{2x_{2n+3}}{x_{n+2}}&-\dfrac{2y_{2n+3}}{x_{n+2}}&\dfrac{x_{2n+4}}{x_{n+2}}&\dfrac{y_{2n+4}}{x_{n+2}}\\
     \vdots&\vdots&\vdots&\vdots\\
     -\dfrac{jx_{2n+3}}{x_{n+j}}&-\dfrac{jy_{2n+3}}{x_{n+j}}&\dfrac{(j-1)x_{2n+4}}{x_{n+j}}&\dfrac{(j-1)y_{2n+4}}{x_{n+j}}\\
     \vdots&\vdots&\vdots&\vdots\\
     -\dfrac{nx_{2n+3}}{x_{2n}}&-\dfrac{ny_{2n+3}}{x_{2n}}&\dfrac{(n-1)x_{2n+4}}{x_{2n}}&\dfrac{(n-1)y_{2n+4}}{x_{2n}}\\[2ex]
     \dfrac{x_{2n+3}}{x_{2n+1}}&\dfrac{y_{2n+3}}{x_{2n+1}}&-\dfrac{x_{2n+4}}{x_{2n+1}}&-\dfrac{y_{2n+4}}{x_{2n+1}}\\[2ex]
     -\dfrac{x_{2n+3}}{x_{2n+2}}&-\dfrac{y_{2n+3}}{x_{2n+2}}&\dfrac{x_{2n+4}}{x_{2n+2}}&\dfrac{y_{2n+4}}{x_{2n+2}}
   \end{pNiceMatrix}\text{\raisebox{-145pt}{.}}
   }
 \end{gather*}
 Therefore, we also have
 \begin{gather*}
   d^2(\tilde{g}\circ\Phi)_{(z)}=A_{\Phi}+B_{\Phi}+{}^t\!B_{\Phi}+C_{\Phi},
 \end{gather*}
 where
 \begin{align*}
   A_{\Phi}&=
   {\addtolength{\arraycolsep}{-3pt}
   \begin{pNiceMatrix}
     \dfrac{H}{|z_1|^2}&R_{1,2}&\cdots&R_{1,2n}&\Vdots&&&&&\Vdots&I_{1,2n+3}&-R_{1,2n+3}&I_{1,2n+4}&-R_{1,2n+4}\\
     R_{1,2}&\dfrac{H}{|z_2|^2}&\cdots&R_{2,2n}& &&&&& &I_{2,2n+3}&-R_{2,2n+3}&I_{2,2n+4}&-R_{2,2n+4}\\
     \vdots&\vdots&\ddots&\vdots& &&&&& &\vdots&\vdots&\vdots&\vdots\\
     R_{1,2n}&R_{2,2n}&\cdots&\dfrac{H}{|z_{2n}|^2}& &&&&& &I_{2n,2n+3}&-R_{2n,2n+3}&I_{2n,2n+4}&-R_{2n,2n+4}\\
     \myhdottedline\\[-10pt]
     &&&& &\Block{2-4}{\raisebox{10pt}{$O^2$}}&&& & &&&&\\
     &&&& &&&& & &&&&\\
     \myhdottedline\\[-10pt]
     I_{1,2n+3}&I_{2,2n+3}&\cdots&I_{2n,2n+3}& &&&&& &\dfrac{H}{|z_{2n+3}|^2}&&-R_{2n+3,2n+4}&\\
     -R_{1,2n+3}&-R_{2,2n+3}&\cdots&-R_{2n,2n+3}& &&&&& &&\dfrac{H}{|z_{2n+3}|^2}&&R_{2n+3,2n+4}\\
     I_{1,2n+4}&I_{2,2n+4}&\cdots&I_{2n,2n+4}& &&&&& &-R_{2n+3,2n+4}&&\dfrac{H}{|z_{2n+4}|^2}&\\
     -R_{1,2n+4}&-R_{2,2n+4}&\cdots&-R_{2n,2n+4}& &&&&& &&R_{2n+3,2n+4}&&\dfrac{H}{|z_{2n+4}|^2}   \end{pNiceMatrix}\text{\raisebox{-95pt}{,}}
   }\\
   B_{\Phi}&=
   {\addtolength{\arraycolsep}{-3pt}
   \begin{pNiceMatrix}
     \Block{2-7}{O^{2n+2}}&&&&&&&\Vdots&&&&\\
     &&&&&&&&&&&\\
     \myhdottedline\\[-10pt]
     &&&&&&&&x_{2n+3}P_{R,2n+3}&y_{2n+3}P_{R,2n+3}&x_{2n+4}Q_{R,2n+3}&y_{2n+4}Q_{R,2n+3}\\
     &&&&&&&&x_{2n+3}P_{I,2n+3}&y_{2n+3}P_{I,2n+3}&x_{2n+4}Q_{I,2n+3}&y_{2n+4}Q_{I,2n+3}\\
     &&&&&&&&x_{2n+3}P_{R,2n+4}&y_{2n+3}P_{R,2n+4}&x_{2n+4}Q_{R,2n+4}&y_{2n+4}Q_{R,2n+4}\\
     &&&&&&&&x_{2n+3}P_{I,2n+4}&y_{2n+3}P_{I,2n+4}&x_{2n+4}Q_{I,2n+4}&y_{2n+4}Q_{I,2n+4}\\
   \end{pNiceMatrix}\text{\raisebox{-32pt}{,}}
   }\\[3ex]
   P_{R,i}&=\sum\limits_{j=1}^n\left(\dfrac{1-j}{x_j}R_{i,j}+\dfrac{j}{x_{n+j}}R_{i,n+j}\right)\text{\raisebox{-6pt}{,}}
   \hspace{5mm}
   Q_{R,i}=\sum\limits_{j=1}^n\left(\dfrac{j}{x_j}R_{i,j}+\dfrac{1-j}{x_{n+j}}R_{i,n+j}\right)\text{\raisebox{-6pt}{,}}\\
   P_{I,i}&=\sum\limits_{j=1}^n\left(\dfrac{1-j}{x_j}I_{i,j}+\dfrac{j}{x_{n+j}}I_{i,n+j}\right)\text{\raisebox{-6pt}{,}}
   \hspace{5mm}
   Q_{I,i}=\sum\limits_{j=1}^n\left(\dfrac{j}{x_j}I_{i,j}+\dfrac{1-j}{x_{n+j}}I_{i,n+j}\right)\text{\raisebox{-6pt}{,}}
   \vspace{5mm}\\[3ex]
   C_{\Phi}&=
   {\addtolength{\arraycolsep}{-3pt}
   \begin{pNiceMatrix}
     \Block{2-7}{O^{2n+2}}&&&&&&&\Vdots&&&&\\
     &&&&&&&&&&&\\
     \myhdottedline\\[-10pt]
     &&&&&&&&x_{2n+3}^2\alpha&x_{2n+3}y_{2n+3}\alpha&x_{2n+3}x_{2n+4}\beta&x_{2n+3}y_{2n+4}\beta\\
     &&&&&&&&x_{2n+3}y_{2n+3}\alpha&y_{2n+3}^2\alpha&y_{2n+3}x_{2n+4}\beta&y_{2n+3}y_{2n+4}\beta\\
     &&&&&&&&x_{2n+3}x_{2n+4}\beta&y_{2n+3}x_{2n+4}\beta&x_{2n+4}^2\gamma&x_{2n+4}y_{2n+4}\gamma\\
     &&&&&&&&x_{2n+3}y_{2n+4}\beta&y_{2n+3}y_{2n+4}\beta&x_{2n+4}y_{2n+4}\gamma&y_{2n+4}^2\gamma
   \end{pNiceMatrix}\text{\raisebox{-32pt}{,}}
     }\\[3ex]
     \alpha&=\sum\limits_{i,j=1}^n\Biggl(\dfrac{j-1}{x_j}\left(\dfrac{i-1}{x_i}(-R_{i,j})+\dfrac{i}{x_{n+i}}R_{n+i,j}\right)+\dfrac{j}{x_{n+j}}\left(\dfrac{i-1}{x_i}R_{i,n+j}+\dfrac{i}{x_{n+i}}(-R_{n+i,n+j})\right)\Biggr)\text{\raisebox{-10pt}{,}}\\[3ex]
     \beta&=\sum\limits_{i,j=1}^n\Biggl(\dfrac{1-j}{x_j}\left(\dfrac{i}{x_i}(-R_{i,j})+\dfrac{i-1}{x_{n+i}}R_{n+i,j}\right)-\dfrac{j}{x_{n+j}}\left(\dfrac{i}{x_i}R_{i,n+j}+\dfrac{i-1}{x_{n+i}}(-R_{n+i,n+j})\right)\Biggr)\text{\raisebox{-10pt}{,}}\\[3ex]
     \gamma&=\sum\limits_{i,j=1}^n\Biggl(\dfrac{j}{x_j}\left(\dfrac{i}{x_i}(-R_{i,j})+\dfrac{i-1}{x_{n+i}}R_{n+i,j}\right)+\dfrac{j-1}{x_{n+j}}\left(\dfrac{i}{x_i}R_{i,n+j}+\dfrac{i-1}{x_{n+i}}(-R_{n+i,n+j})\right)\Biggr)\text{\raisebox{-10pt}{.}}
 \end{align*}
 Note that in the above equations for $\alpha, \beta, \gamma$, we set $-R_{i,i}=\dfrac{H}{|z_i|^2}$.
 The pullback of $\omega_{\text{std}}$ is given by
 \begin{gather*}
 (\Phi^*\omega_{\text{std}})_{\Phi^{-1}(z)}=
 {\addtolength{\arraycolsep}{-3pt}
 \begin{pNiceMatrix}
   \Block{2-7}{O^{2n+2}}&&&&&&&\Vdots&\Block{2-4}{-B_{\varphi}}&&&\\
   &&&&&&& &&&&\\
   \myhdottedline\\[-10pt]
   \Block{4-7}{\raisebox{10pt}{${}^t\!B_{\varphi}$}}&&&&&&& &0&1&&\\
   &&&&&&& &-1&0&&\\
   &&&&&&& &&&0&1\\
   &&&&&&& &&&-1&0
 \end{pNiceMatrix}\text{\raisebox{-35pt}{.}}
 }
 \end{gather*}
 \item The subspace $L_{[z]}\subset T_{[z]}M_n$ is spanned by $X_J([z])$. Now, consider the following three matrices $(B_a)_1,\allowbreak (B_a)_2,(B_a)_3$:\vspace{3mm}
 \begin{align*}
 (B_a)_1&=
 {\addtolength{\arraycolsep}{-4pt}
 \begin{pNiceMatrix}
   x_1&&&&&\\
   &x_2&&&&\\
   &&\ddots&&&\\
   &&&x_j&&\\
   &&&&\ddots&\\
   &&&&&x_n\\
   \myhdottedline\\[0pt]
   \Block{1-6}{O_{n+1,n}}&&&&&\\
   &&&&&\\
   \myhdottedline\\[-10pt]
   0&x_{2n+2}&\cdots&(j-1)x_{2n+2}&\cdots&(n-1)x_{2n+2}\\
   \myhdottedline\\[0pt]
   \Block{1-6}{O_{2,n}}&&&&&\\
   &&&&&\\
   \myhdottedline\\[-10pt]
   -y_{2n+4}&-y_{2n+4}&\cdots&-y_{2n+4}&\cdots&-y_{2n+4}\\
   x_{2n+4}&x_{2n+4}&\cdots&x_{2n+4}&\cdots&x_{2n+4}\\
 \end{pNiceMatrix}\text{\raisebox{-92pt}{,}}
 }\\[3ex]
 (B_a)_2&=
 {\addtolength{\arraycolsep}{-4pt}
 \begin{pNiceMatrix}
  \Block{2-8}{O_{n,n+2}}&&&&&&&\\
   &&&&&&&\\
   \myhdottedline\\[-10pt]
   x_{n+1}&&&&&&&\\
   &x_{n+2}&&&&&&\\
   &&\ddots&&&&&\\
   &&&x_{n+j}&&&&\\
   &&&&\ddots&&&\\
   &&&&&x_{2n}&&\\
   0&x_{2n+1}&\cdots&(j-1)x_{2n+1}&\cdots&(n-1)x_{2n+1}&x_{2n+1}&x_{2n+1}\\
   &&&&&&x_{2n+2}&\\
   -y_{2n+3}&-y_{2n+3}&\cdots&-y_{2n+3}&\cdots&-y_{2n+3}&&y_{2n+3}\\
   x_{2n+3}&x_{2n+3}&\cdots&x_{2n+3}&\cdots&x_{2n+3}&&-x_{2n+3}\\
   &&&&&&&-y_{2n+4}\\
   &&&&&&&x_{2n+4}
 \end{pNiceMatrix}\text{\raisebox{-86pt}{,}}
 }\\[3ex]
 (B_a)_3&=
 {\addtolength{\arraycolsep}{-4pt}
 \begin{pNiceMatrix}
   \Block{4-4}{O_{2n+2,4}}&&&\\
   &&&\\
   &&&\\
   &&&\\
   \myhdottedline\\[-10pt]
   y_{2n+3}&x_{2n+3}&&\\
   -x_{2n+3}&y_{2n+3}&&\\
   -y_{2n+4}&&x_{2n+4}&-y_{2n+4}\\
   x_{2n+4}&&y_{2n+4}&x_{2n+4}
 \end{pNiceMatrix}\text{\raisebox{-43pt}{,}}
 }\\
 \end{align*}
 Then, the column vectors of the matrix $B_a=\bigl((B_a)_1\,(B_a)_2\,(B_a)_3\bigr)$ form a basis for $T_{\Phi^{-1}(z)}\mathbb{R}^{2n+6}$. Moreover, the column vectors excluding those of $(B_a)_3$ form a basis for $d\Phi^{-1}(X_K)_z$, and the first column vector of $(B_a)_3$ is $X_{\tilde{J}}(z)$.
With respect to this basis of $T_{\Phi^{-1}(z)}\mathbb{R}^{2n+6}$, the $4$-by-$4$ submatrices, $d^2G$ and $\Omega$, from the $(2n+3,2n+3)$ to $(2n+6,2n+6)$ entries of the matrices representing $d^2(\tilde{g}\circ\Phi)_{\Phi^{-1}(z)}$ and $(\Phi^*\omega_{\text{std}})_{\Phi^{-1}(z)}$, respectively, are given as follows:
\begin{gather*}
d^2G=A_G+B_G+{}^tB_G+C_G,
\end{gather*}
where\vspace{3mm}
\begin{gather*}
 A_G=
 {\addtolength{\arraycolsep}{-2pt}
 \begin{pmatrix}
   0&&&\\
   &H&-H&\\
   &-H&H&\\
   &&&H
 \end{pmatrix}
 }\text{\raisebox{-18pt}{,}}\\[3ex]
 B_G=
 {\addtolength{\arraycolsep}{-2pt}
 \begin{pmatrix}
   0&&&\\
&|z_{2n+3}|^2H\sum\limits_{j=1}^n\left(\dfrac{1-j}{|z_j|^2}+\dfrac{j}{|z_{n+j}|^2}\right)&|z_{2n+4}|^2H\sum\limits_{j=1}^n\left(\dfrac{j}{|z_j|^2}+\dfrac{1-j}{|z_{n+j}|^2}\right)&\\
   &|z_{2n+3}|^2H\sum\limits_{j=1}^n\left(\dfrac{1-j}{|z_j|^2}+\dfrac{j}{|z_{n+j}|^2}\right)&|z_{2n+4}|^2H\sum\limits_{j=1}^n\left(\dfrac{j}{|z_j|^2}+\dfrac{1-j}{|z_{n+j}|^2}\right)&\\
   &&&0
 \end{pmatrix}\text{\raisebox{-30pt}{,}}
 }\\[3ex]
 C_G=
 {\addtolength{\arraycolsep}{-2pt}
  \begin{pmatrix}
    0&&&\\
    &|z_{2n+3}|^4\alpha&|z_{2n+3}|^2|z_{2n+4}|^2\beta&\\
    &|z_{2n+3}|^2|z_{2n+4}|^2\beta&|z_{2n+4}|^2\gamma&\\
    &&&0
  \end{pmatrix}
  }\text{\raisebox{-17pt}{,}}
  \\[3ex]
  \Omega=
  {\addtolength{\arraycolsep}{-2pt}
  \begin{pmatrix}
    &|z_{2n+3}|^2&-|z_{2n+4}|^2&\\
    -|z_{2n+3}|^2&&&\\
    |z_{2n+4}|^2&&&|z_{2n+4}|^2\\
    &&-|z_{2n+4}|^2&
  \end{pmatrix}\text{\raisebox{-20pt}{.}}
  }\\
\end{gather*}

\item Consider the following matrix:\vspace{3mm}
\begin{gather*}
  {\addtolength{\arraycolsep}{-2pt}
  \begin{pmatrix}
    1&&&\\
    &&|z_{2n+4}|^2&-|z_{2n+3}|^2\\
    &&|z_{2n+3}|^2&|z_{2n+4}|^2\\
    &1&&
  \end{pmatrix}\text{\raisebox{-20pt}{.}}
  }\\
\end{gather*}
While the set of column vectors of this matrix forms a basis of $\mathbb{R}^4$, the first three column vectors form a basis of $L_{[z]}^{\omega}$ and the first column vector spans $L_{[z]}$.
The submatrices consisting of the $(2,2)$ to $(3,3)$ entries of the matrices obtained by pulling back $d^2G$ and $\Omega$ by this basis are given as follows:
\begin{gather*}
\widehat{d^2G}=
  {\addtolength{\arraycolsep}{-2pt}
  \begin{pmatrix}
    H&0\\
    0&c
  \end{pmatrix}\text{\raisebox{-7pt}{,}}
  }\hspace{5mm}
  \widehat{\Omega}=
  {\addtolength{\arraycolsep}{-2pt}
  \begin{pmatrix}
    0&-|z_{2n+3}|^2|z_{2n+4}|^2\\
    |z_{2n+3}|^2|z_{2n+4}|^2&0
  \end{pmatrix}\text{\raisebox{-7pt}{,}}
  }
\end{gather*}
where 
\begin{equation*}
c=2H|z_{2n+3}|^4|z_{2n+4}|^4\left(\dfrac{1}{|z_{2n+3}|^4}+\dfrac{1}{|z_{2n+4}|^4}+\sum_{k=1}^{2n}\dfrac{1}{|z_{k}|^4}\right)\text{\raisebox{-8pt}{.}} 
\end{equation*}
Then the characteristic equation $P(t)$ is 
\begin{equation*}
P(t)=t^2+2\left(\dfrac{1}{|z_{2n+3}|^4}+\dfrac{1}{|z_{2n+4}|^4}+\sum_{k=1}^{2n}\dfrac{1}{|z_{k}|^4}\right)H^2.
\end{equation*}
\item Therefore, the singularity $[p]$ is an Elliptic-Regular type.
\end{enumerate}

\bibliographystyle{amsplain} 
\bibliography{references}
\end{document}